\documentclass{article}
\usepackage{amsmath,amsfonts,epsfig}
\usepackage{amsmath,amsfonts,amssymb}
\usepackage{booktabs}
\usepackage{amsmath,amsfonts,epsfig}
\usepackage{graphics}
\usepackage{supertabular}
\usepackage{amsmath,amsfonts,amssymb}
\usepackage{amsthm}
\usepackage{mathtools}
\usepackage{amsfonts}
\usepackage{cite}
\usepackage{graphicx}
\usepackage{epstopdf}
\usepackage[numbers]{natbib}
\usepackage{float,epsfig, floatflt,here}
\usepackage{color}
\usepackage[colorlinks=true,citecolor=blue,linkcolor=blue]{hyperref}
\usepackage{fancyhdr}
\usepackage{etoolbox}
\usepackage{ dsfont }
\usepackage{accents}

\newcommand{\norm}[1]{\| #1 \|}
\newcommand{\HHO}{\mathrm{HHO}}
\newcommand{\hhonorm}[1]{\lVert #1 \rVert_{\HHO}}
\newcommand{\tnorm}[1]{{\left\vert\kern-0.25ex\left\vert\kern-0.25ex\left\vert #1\right\vert\kern-0.25ex\right\vert\kern-0.25ex\right\vert}}

\newcommand{\vertiii}[1]{{\left\vert\kern-0.25ex\left\vert\kern-0.25ex\left\vert #1
		\right\vert\kern-0.25ex\right\vert\kern-0.25ex\right\vert}}

\newtheorem{thm}{Theorem}[section]

\newtheorem{defn}{Definition}[section]

\newtheorem{lemma}{Lemma}[section]

\newtheorem{rem}{Remark}[section]

\numberwithin{figure}{section}
\numberwithin{table}{section}
\usepackage{subcaption}
\usepackage{orcidlink}
\title{\normalsize{Energy norm error estimates of a hybrid high-order method for the linear parabolic integro-differential equations on general meshes}}
\date{}
\author{
	{\normalsize Achyuta Ranjan Dutta Mohapatra \orcidlink{0009-0001-1048-6804}}\thanks{Corresponding author, Department of Mathematics,
		Indian Institute of Technology Guwahati, North Guwahati- 781039, India ({\tt achyutar@iitg.ac.in}).}}

\date{}
\begin{document}
	\maketitle
	\begin{abstract}	
		We are concerned in designing a suitable numerical scheme based on the equal-order hybrid high-order (HHO) method for the linear parabolic integro-differential equations. The spatial discretization is made using the equal-order HHO method and subsequently we perform the stability analysis of the corresponding semi-discrete scheme. The convergence results are presented in suitably defined Bochner norms for the semi-discrete problem. Then a second-order temporal discretization is implemented on the time domain using a Crank-Nicolson scheme where the memory term is approximated using a composite trapezoidal quadrature rule. The stability of the resultant complete discrete schemes are analyzed followed by derivation of the error estimates of order $\mathcal{O}(\tau^{2}+h^{k+1})$, $k\ge 0$ is the degree of local polynomial approximation, in discrete $l^{2}(0,T;H^{1}(\Omega))$ and $l^{\infty}(0,T;H^{1}(\Omega))$ like norms. Numerical illustrations are performed on some polygonal meshes validating the theoretical estimates.
	\end{abstract}	
	\vspace{.01in}
	
	
	\section{\normalsize Introduction}\label{sec1}
	In this article, we are concerned about designing a high-order accurate numerical solution of the following linear parabolic integro-differential equations (PIDEs) posed in the domain $\Omega\times(0,T)$ given as:
	\begin{eqnarray}\label{model}
		\left\{
		\begin{array}{ll}
			p_{t}({\bf x},t)-\Delta p({\bf x},t)-\int_{0}^{t}\Delta p({\bf x},s)\,ds = f({\bf x},t)\,\,\mathrm{in}\,\,\Omega\times(0,T),\\
			p({\bf x},t) = 0 \,\,\mathrm{on}\,\,\partial\Omega\times(0,T),\\
			p({\bf x},0)=g({\bf x}).
		\end{array}
		\right.
	\end{eqnarray}
	Here we have assumed $\Omega\subset\mathbb{R}^{2}$ is a bounded convex polygonal domain with its boundary denoted as $\partial\Omega$. Further, it is assumed that the initial function ($g$) and the forcing term ($f$) are in the spaces $H^{1}(\Omega)$ and $L^{2}(\Omega)$, respectively.
	
	The weak formulation for the equation \eqref{model} is described as: Find $u({\bf x},t)\in H^{1}_{0}(\Omega)$ $(t>0)$ with $u({\bf x},0)=g({\bf x})$ such that
	\begin{eqnarray}\label{vf}
		(p_{t}(t),w)+(\nabla p(t),\nabla w)+\int_{0}^{t}(\nabla p(s),\nabla w)\,ds=(f(t),w)\,\forall w\in H^{1}_{0}(\Omega).
	\end{eqnarray}
	By abuse of notation, throughout the article we denote $u({\bf x},t):=u(t)$ for any function $u:\Omega\times(0,T)\to\mathbb{R}$.
	

The PIDEs have significant applications in rigid heat conduction with memory (cf. \cite{miller1978integrodifferential}), nuclear reactor theory (cf. \cite{pao1974solution}), biotechnology (cf. \cite{allegretto1999box}), etc. 
Some early works on the spatial discretization based on continuous Galerkin finite element methods (CGFEMs) for the linear PIDEs include \cite{thomee1989error,greenwell1988finite,sloan1986time,lin1991ritz,pani1996finite,pani1996superconvergence} which incorporate backward Euler (BE) time discretization scheme. Unlike the standard heat equations, the presence of the memory term in the PIDEs creates additional hurdles in error analysis while employing time-discretization schemes. Although the convergence analysis of the complete discrete BE time-stepping combined with FEM-type discretization in space for the linear PIDEs is abundant, the literature on Crank-Nicolson (CN) FEMs for linear PIDEs is limited. In \cite{sloan1986time}, the authors presented the stability and temporal error estimate of a time-discrete CN scheme for the PIDE. Later in \cite{yun1994time}, the regularity assumptions on the true solution were relaxed while discussing the temporal convergence analysis. A discontinuous Galerkin (DG) time-stepping combined with CGFEM in space was presented in \cite{larsson1998numerical,mustapha2011hp}. Recently, in \cite{xu2020superclose}, a CN-based two-grid FEM for the semilinear PIDEs was presented, and the superclose error analysis was discussed. Semi-discrete mixed FEMs with non-smooth initial data were proposed in \cite{sinha2009mixed}, and BE mixed FEM were presented in \cite{liu2015new}.       
A BE-based $H^{1}$-mixed FEM was put forward in \cite{pani2002h}. More recently, in \cite{xu2024error}, a BE serendipity virtual element method (VEM) for the semilinear PIDE was presented, and in \cite{xu2026error}, the same authors extended the time discretization to the CN scheme with serendipity VEM in space on curved domains. The authors in \cite{pani2011hp} presented a local DG method for PIDEs, and in \cite{jain2023hdg}, a BE-based hybridizable DG method was presented. Time discretization based on the backward Euler method combined with Weak Galerkin (WG) FEMs was discussed in \cite{wang2020weak,zhu2016weak}. However, to the best of our knowledge, there are no articles in the literature that employ hybrid high-order methods for spatial discretization in PIDEs.
     
     The core of the HHO method lies in approximating the exact solution by its discrete counterpart at both mesh cells and faces, separately, using polynomials of appropriate degrees. If polynomials employed for solution approximation are of the same degree in mesh cells and faces, the corresponding HHO method is known to be the equal-order HHO method. The standard differential operators are replaced by their locally reconstructed counterparts acting on cell and face unknowns. High-order stabilization is generally incorporated to penalize the difference between the traces of cell and face unknowns. The stabilization principles employed in equal-order HHO are inherently different than those of WG and HDG methods (see \cite{ern2024convergence}). These methods also support polygonal meshes, local conservation properties, and for smooth solutions, they provide convergence rates of order $\mathcal{O}(h^{k+1})$ and $\mathcal{O}(h^{k+2})$, $k\ge0$ is the degree of local polynomial approximation, in appropriately defined energy and $L^{2}$-norms, respectively, which in fact is one order higher than some other variants of finite element methods like CG-FEM, WG-FEM, DG-FEM and VEM. They also support static condensation, in which the global system, with both cell and face unknowns, is reduced in determining only the face unknowns.       
     These methods have been successfully implemented across various types of partial differential equations (see e.g. \cite{di2014arbitrary,botti2017hybrid,ern2024convergence,cicuttin2020hybrid,ern2020quasi,di2017hybrid,mottier2026elasto,gudi2022hybrid,singh2025high} and the references therein). The textbooks \cite{di2020hybrid,cicuttin2021hybrid} offer excellent theory on the basics of HHO methods, their variants, and their applications to other partial differential equations. 
    
    The contributions of this work as detailed as follows:
    \begin{itemize}
    	\item We perform the space discretization of the PIDE \eqref{model} by employing an equal-order HHO method and establish the stability of the corresponding scheme.
    
    \item After deriving the error equation, we present the convergent rates of $\mathcal{O}(h^{k+1})$, $k\ge0$ is the degree of local polynomial approximation, for the semi-discrete problem.
    
    \item The complete discrete HHO method is then presented, where a Crank-Nicolson discretization is employed in the time-direction, and the memory term is approximated using the composite trapezoidal rule. The stability of this scheme is further analyzed.
    
    \item We then derive the error estimates of $\mathcal{O}(h^{k+1}+\tau^{2})$  from the consistency error equation under a discrete $l^{2}(0,T;H^{1}(\Omega))$ like norm and as a by-product, we also obtain same convergence rates under a discrete $l^{\infty}(0,T;L^{2}(\Omega))$.
    
    \item By choosing an appropriate test function in the the consistency error equation we also derive the convergence results under the discrete $l^{\infty}(0, T;H^{1}(\Omega))$ like norm. Some numerical assessments are performed on polygonal meshes, thereby justifying the optimality of the proposed schemes. 
    \end{itemize}
    
    This paper is organized in the following manner: The preliminaries of the HHO methods are discussed in Section \ref{sec3}. The semi-discrete (spatially discretized) problem is formulated in Section \ref{sec4}, followed by the derivation of its stability and error estimates. Section \ref{sec5} discusses the stability and convergence of the complete discrete Crank-Nicolson HHO scheme for the linear PIDE. Numerical assessments are done in Section \ref{sec6} for varying mesh structures, and Section \ref{sec7} concludes this article.

	\section{\normalsize The HHO Method and its preliminaries}\label{sec3}
	
	In this section, we present some basic prerequisites for introducing the HHO methods, primarily based on the article \cite{ern2024convergence} and the book \cite{di2020hybrid}. The standard notations of Sobolev and Bochner spaces are used in this article. We have used the notation $C$ throughout this article to denote a positive constant whose value varies with context, is independent of the space and time mesh sizes, but may depend on the final time $T$.
	
	Let $\mathcal{K}_{h}$, with $h>0$, denote a tessellation of the spatial domain $\Omega$ into a finite number of cells satisfying some mesh regularity constraints as detailed in \cite{di2020hybrid} composed of closed and connected polygons. Suppose that for a cell $K\in\mathcal{K}_{h}$, let $h_{K}$ denotes its diameter. Then the mesh size corresponding to $\mathcal{K}_{h}$ is given as $h=\max_{K\in\mathcal{K}_{h}}h_{K}$. Denote by $\mathcal{F}_{h}$ as the collection of all faces (edges) of cells in $\mathcal{K}_{h}$ with $\mathcal{F}_{h}=\mathcal{F}_{h}^{o}\cup\mathcal{F}_{h}^{\partial}$, where $\mathcal{F}_{h}^{o}$ is the set of all intra-cell mesh faces (interfaces) and $\mathcal{F}_{h}^{\partial}$ is the set of all boundary faces. Now, for each cell $K\in\mathcal{K}_{h}$, let $\mathcal{F}_{K}$ denote the set of all the faces $F$ contained on the boundary $\partial K$ of the cell $K$ and let $h_{F}$ denote the length of the face $F$. For a cell $K\in\mathcal{K}_{h}$, denote by ${\bf n}_{KF}$ as the unit outward normal to $K$ along the face (say) $F$ of $K$.
	
	For the HHO methods, on each cell of the mesh partition $\mathcal{K}_{h}$, the unknowns are associated with polynomials of suitable degrees constructed on both the cell interiors as well as the faces, which approximate the exact solution. Let $k\ge 0 $ be an integer. For an arbitrary cell $K$, we introduce the following polynomial spaces $\mathcal{P}_{k}(K)$ and $\mathcal{P}_{k}(F)$ of degree $k$, where $F\in\mathcal{F}_{K}$ i.e. $F$ denotes a face associated with the cell $K$.
	
	The linear space incorporating all the cell and face degrees of freedom is denote by $V_{h}^{k}$ and is given as
	\begin{eqnarray*}
		V_{h}^{k}:=\prod_{K\in\mathcal{K}_{h}}\mathcal{P}_{k}(K)\times\prod_{F\in\mathcal{F}_{h}}\mathcal{P}_{k}(F).
	\end{eqnarray*}
	Here the Cartesian products $\prod_{K\in\mathcal{K}_{h}}\mathcal{P}_{k}(K)$ and $\prod_{F\in\mathcal{F}_{h}}\mathcal{P}_{k}(F)$ represent the complete cells and faces degrees of freedom, respectively. Further, a broken piece-wise polynomial function described by $q_{\mathcal{K}}=(q_{K})_{K\in\mathcal{K}_{h}}\in \prod_{K\in\mathcal{K}_{h}}\mathcal{P}_{k}(K)$ is defined a.e. over the domain $\Omega$ by interpreting $q_{\mathcal{K}}|_{K}:=q_{K}$ $\forall\, K\in\mathcal{K}_{h}$. In a similar manner, we can define $q_{\mathcal{F}}=(q_{F})_{F\in\mathcal{F}_{h}}\in \prod_{F\in\mathcal{F}_{h}}\mathcal{P}_{k}(F)$ is defined a.e. over the domain $\Omega$ by interpreting $q_{\mathcal{F}}|_{F}:=q_{F}$ $\forall\, F\in\mathcal{F}_{h}$. 
	
	Motivated from the above discussed notations, any element $\hat{q}_{h}\in V_{h}^{k}$ is expressed as $\hat{q}_{h}:=(q_{\mathcal{K}},q_{\mathcal{F}})\in\prod_{K\in\mathcal{K}_{h}}\mathcal{P}_{k}(K)\times\prod_{F\in\mathcal{F}_{h}}\mathcal{P}_{k}(F)$. Henceforth, throughout this article, any variables with a hat refer to the hybrid functions.
	
	Further on any arbitrary cell (say) $K\in\mathcal{K}_{h}$, the local hybrid finite element space composed of unknowns in both the respective cell and the corresponding faces is expressed as $V_{K}^{k}:=\mathcal{P}_{k}(K)\times\prod_{F\in\mathcal{F}_{K}}\mathcal{P}_{k}(F)$. Hence, a local hybrid function can be expressed as $\hat{q}_{K}:=(q_{K},q_{\partial K}:=(q_{F})_{F\in\mathcal{F}_{K}})\in V_{K}^{k}$. By the abuse of notation, we interpret the element $q_{\partial K}|_{F}:=q_{F}$ $\forall F\in\mathcal{F}_{K}$.
	
	To incorporate the homogeneous Dirichlet conditions, we introduce the following subspace of face unknowns given by
	\begin{equation*}
		\mathbb{F}_{h,0}^{k}:=\left\{q_{\mathcal{F}}\in\prod_{F\in\mathcal{F}_{h}}\mathcal{P}_{k}(F):q_{F}=0,\,\forall F\in \mathcal{F}_{h}^{\partial}\right\}.
	\end{equation*}
	Hence, the hybrid finite element space inclusive of the Dirichlet boundary data is given by:
	\begin{equation*}
		V_{h,0}^{k}:=\prod_{K\in\mathcal{K}_{h}}\mathcal{P}_{k}(K)\times \mathbb{F}_{h,0}^{k}.
	\end{equation*}
	For the sake of brevity, we can express the global hybrid space in the following way:
	\begin{eqnarray*}
		&&	V_h^k := \Bigl\{ 
		\hat{q}_h = (q_{\mathcal K}=(q_{K})_{K\in\mathcal{K}_{h}}, q_{\mathcal F} = (q_F)_{F\in\mathcal F_h})\in \prod_{K\in\mathcal{K}_{h}}\mathcal{P}_{k}(K)\times\prod_{F\in\mathcal{F}_{h}}\mathcal{P}_{k}(F) \;\Big|\;
		\nonumber\\
		&&\qquad\qquad q_{\mathcal K}|_K = q_K \in \mathcal P_k(K), \;
		q_F \in \mathcal P_k(F)\ \forall F\in\mathcal F_K,\, \forall K\in\mathcal K_h
		\Bigr\}.
	\end{eqnarray*}
	This implies that $\hat{q}_{h}|_{K}=\hat{q}_{K}$.
	The same can be done for the space $V_{h,0}^{k}$.

	Next we introduce some local $L^{2}$-projections, which are denoted in some variants of symbol $\mathbb{P}$, defined on each cell $K\in\mathcal{K}_{h}$ with $F\in\mathcal{F}_{K}$ be any of its faces. Then we present the following operators: 
	\begin{itemize}
		\item $\pi^{k}_{K}:L^{2}(K)\to \mathcal{P}_{k}(K)$,
		
		\item $\pi^{k}_{\partial K}:L^{2}(\partial K)\to \prod_{F\in\mathcal{F}_{K}}\mathcal{P}_{k}(F)$,
		
		\item $\pi^{k}_{F}:L^{2}(F)\to \mathcal{P}_{k}(F)$.
	\end{itemize}
	Here the operator $\pi^{k}_{\partial K}$ is defined component-wise by $(\pi^{k}_{\partial K}w)|_{F}=\pi^{k}_{F}(w|_{F})$ with $w\in L^{2}(\partial K)$.
	Further, the corresponding global $L^{2}$-projections are given as:
	\begin{itemize}
		\item $\pi^{k}_{\mathcal{K}}:L^{2}(\Omega)\to \prod_{K\in\mathcal{K}_{h}}\mathcal{P}_{k}(K)$,
		
		\item $\pi^{k}_{\mathcal{F}}:\prod_{F\in\mathcal{F}_{h}} L^{2}(F)\to \prod_{F\in\mathcal{F}_{h}}\mathcal{P}_{k}(F)$.
	\end{itemize}
	Here the operator $\pi^{k}_{\mathcal{K}}$ is defined component-wise by $(\pi^{k}_{\mathcal{K}}w)|_{K}=\pi^{k}_{K}(w|_{K})$ with $w\in L^{2}(\Omega)$. Similarly, the operator $\pi^{k}_{\mathcal{F}}$
	defined component-wise as follows: for any
	$w = (w_F)_{F\in\mathcal F_h} \in \prod_{F\in\mathcal F_h} L^2(F)$,
	$(\pi_{\mathcal{F}}^k w)|_F := \pi_F^k(w|_F)$ $\forall F\in\mathcal F_h$.
	
	From the above notions we introduce a local projection operator $\hat{I}_{K}^{k}:H^{1}(K)\to V_{K}^{k}$, for all $K\in\mathcal{K}_{h}$, given
	\begin{equation*}
		\hat{I}_{K}^{k}(v):=(\pi_{K}^{k}(v),\pi_{\partial K}^{k}(v))\,\forall v\in H^{1}(K).
	\end{equation*}
	The global projection operator is further given by $\hat{I}_{h}^{k}:H^{1}(\Omega)\to V_{h}^{k}$ with
	\begin{equation*}
		\hat{I}_{h}^{k}(v):=(\pi_{\mathcal{K}}^{k}(v),\pi_{\mathcal{F}}^{k}(v))=((\pi_{K}^{k}(v))_{K\in\mathcal{K}_{h}},(\pi_{F}^{k}(v))_{F\in\mathcal{F}_{h}}).
	\end{equation*}
	In a similar manner, we can map $\hat{I}_{h}^{k}:H_{0}^{1}(\Omega)\to V_{h,0}^{k}$.
	
	We denote $(\cdot,\cdot)_{K}$ and $\langle\cdot,\cdot\rangle_{\partial K}$ and $\langle\cdot,\cdot\rangle_{F}$ as the $L^{2}$ inner products associated with the cell $K$, its boundary $\partial K$ and corresponding face $F$, respectively, for any $K\in\mathcal{K}_{h}$ and $F\in\mathcal{F}_{K}$. Also, we have $\norm{\cdot}_{K}$, $\norm{\cdot}_{\partial K}$ and $\norm{\cdot}_{F}$ are the associated norms with the inner products $(\cdot,\cdot)_{K}$ and $\langle\cdot,\cdot\rangle_{\partial K}$ and $\langle\cdot,\cdot\rangle_{F}$, respectively.
	
	Now, we introduce the notion potential reconstruction operator (see Section 2.1.3 of \cite{di2020hybrid}).
	\begin{defn}[The potential reconstruction]\label{potentialreconstruction}	On a given cell $K\in\mathcal{K}_{h}$, the potential reconstruction operator $R_{K}^{k+1}:V^{k}_{K}\to \mathcal{P}_{k+1}(K)$ is a unique computable quantity derived from solving the local Neumann problem:
		\begin{eqnarray*}
			\left\{
			\begin{array}{ll}
				(\nabla R_{K}^{k+1}(\hat{q}_{K}),\nabla w)=(\nabla q_{K},\nabla w)+\sum_{F\in\mathcal{F}_{K}}\langle q_{F}-q_{K}|_{F},{\bf n}_{KF}\cdot\nabla w\rangle_{F}\,\forall w\in \mathcal{P}_{k+1}(K),\\
				(R_{K}^{k+1}\hat{q}_{K},1)_{K}=(q_{K},1).
			\end{array}
			\right.
		\end{eqnarray*}
	\end{defn}
	The global reconstruction operator is defined as: $R_{h}^{k+1}:V_{h}^{k}\to \mathcal{P}_{k+1}(\mathcal{K}_{h})$ $(R_{h}^{k+1}(\hat{q}_{h}))|_{K}=R_{K}^{k+1}(\hat{q}_{K})$, $\forall \hat{q}_{h}=(q_{\mathcal{K}},q_{\mathcal{F}})\in V_{h}^{k}$. Here $\mathcal{P}_{k+1}(\mathcal{K}_{h})$ is the space of all broken polynomials of local degree $k+1$ defined on the mesh partition $\mathcal{K}_h$ i.e. the space contains elements which locally polynomials of degree $k+1$ on each cell (say) $K$ of the mesh partition $\mathcal{K}_{h}$.
	
	Since we are approximating the exact solution using discontinuous functions at each cell and face, a stabilization is necessary to enforce the required continuity constraints. Let $K\in\mathcal{K}_{h}$ be any cell and let $F\in\mathcal{F}_{K}$ be an arbitrary face of cell $K$. For any $\hat{q}_{K}\in V_{K}^{k}$, we define the operators $\delta_{K}:V_{K}^{k}\to \mathcal{P}_{k}(K)$ and $\delta_{KF}:V_{K}^{k}\to \mathcal{P}_{k}(F)$ such that for all $\hat{q}_{K}\in V_{K}^{k}$,
	\begin{eqnarray}
		\delta_{K}(\hat{q}_{K}):=\pi_{K}^{k}(R_{K}^{k+1}(\hat{q}_{K})-q_{K})\,\mathrm{and}\,\delta_{KF}(\hat{q}_{K}):=\pi_{F}^{k}(R_{K}^{k+1}(\hat{q}_{K})-q_{F}),\forall F\in\mathcal{F}_{K}.
	\end{eqnarray}
	Motivated from the above definitions, we introduce the equal-order HHO stabilization which is a semi-positive definite bilinear form $s_{K}:V_{K}^{k}\times V_{K}^{k}$ defined in the following manner.
	\begin{eqnarray*}
		s_{K}(\hat{q}_{K},\hat{r}_{K}):=\sum_{F\in\mathcal{F}_{K}}h_{F}^{-1}\langle(\delta_{KF}-\delta_{K})\hat{q}_{K},(\delta_{KF}-\delta_{K})\hat{r}_{K}\rangle_{F}.
	\end{eqnarray*}
	For more details on HHO stabilizations, one can refer to Section 2.1.4 of \cite{di2020hybrid}.
	
	We now introduce the discrete bilinear form. For $\hat{q}_{h},\hat{r}_{h}\in V_{h}^{k}$, define a map $b(\cdot,\cdot):V_{h}^{k}\times V_{h}^{k}\to \mathbb{R}$ by
	\begin{eqnarray}\label{bilinear}
		b(\hat{q}_{h},\hat{r}_{h}):=a_{h}(\hat{q}_{h},\hat{r}_{h})+s_{h}(\hat{q}_{h},\hat{r}_{h}),
	\end{eqnarray}
	where $a_{h}(\cdot,\cdot),s_{h}(\cdot,\cdot):V_{h}^{k}\times V_{h}^{k}\to \mathbb{R}$ are defined as
	\begin{eqnarray*}
		a_{h}(\hat{q}_{h},\hat{r}_{h}):=\sum_{K\in\mathcal{K}_{h}}a_{K}(\hat{q}_{K},\hat{r}_{K}),\quad\,s_{h}(\hat{q}_{h},\hat{r}_{h}):= \sum_{K\in\mathcal{K}_{h}}s_{K}(\hat{q}_{K},\hat{r}_{K}),
	\end{eqnarray*}
	with
	\begin{eqnarray*}
		a_{K}(\hat{q}_{K},\hat{r}_{K})=(\nabla R_{K}^{k+1}(\hat{q}_{K}),\nabla R_{K}^{k+1}(\hat{r}_{K}))_{K}.
	\end{eqnarray*}
	From the above-defined bilinear form \eqref{bilinear}, we can generate a semi-norm on the space $V_{h}^{k}$ (resp. a norm on the space $V_{h,0}^{k}$) given by 
	\begin{equation}\label{bilnorm}
		\norm{\hat{q}_{h}}_{1,h}^{2}:=b(\hat{q}_{h},\hat{q}_{h}), \,\forall \hat{q}_{h}\in V_{h}^{k}.
	\end{equation}
	Further, the functional $\hhonorm{\cdot}:V_{h}^{k}\to\mathbb{R}_{+}$ given by
	\begin{eqnarray*}
		\hhonorm{\hat{q}_{h}}^{2}:=\sum_{K\in\mathcal{K}_{h}}\norm{\hat{q}_{K}}_{1,K}^{2}
	\end{eqnarray*}
	defines a norm on the linear space $V_{h,0}^{k}$ (see Corollary 2.16. of \cite{di2020hybrid}), where,
	\begin{eqnarray*}
		\norm{\hat{q}_{K}}_{1,K}^{2}:=\norm{\nabla q_{K}}_{L^{2}(K)}^{2}+\sum_{F\in\mathcal{F}_{K}}h_{F}^{-1}\norm{q_{F}-q_{K}|_{F}}_{L^{2}(F)}^{2}.
	\end{eqnarray*}
	
	The norms $\hhonorm{\cdot}$ and $\norm{\cdot}_{1,h}$ on the space $V_{h,0}^{k}$ are equivalent (see Lemma 2.18 of \cite{di2020hybrid}) i.e. there exists a positive constant $C_{eqv}$ independent of $h$ such that
	\begin{equation}\label{equivalence}
		C_{eqv}\hhonorm{\hat{q}_{h}}\le \norm{\hat{q}_{h}}_{1,h}\le C_{eqv}^{-1}\hhonorm{\hat{q}_{h}}.
	\end{equation}
	Also, we can define the stabilization semi-norm on the space $V_{h}^{k}$ in the following way:
	\begin{equation*}
		|\hat{q}_{h}|_{*,h}^{2}:=s_{h}(\hat{q}_{h},\hat{q}_{h}),\,\forall \hat{q}_{h}\in V_{h}^{k}.
	\end{equation*}  
	By Cauchy-Schwarz inequality, we can establish the boundedness property of the bilinear form b($\cdot$,$\cdot$) given as
	\begin{equation}\label{boundedness}
		b(\hat{q}_{h},\hat{r}_{h})\le \norm{\hat{q}}_{1,h}\norm{\hat{r}}_{1,h}.
	\end{equation} 
	\begin{lemma}[Discrete local trace inequality, Lemma 3.1 of \cite{ern2024convergence}]\label{trace}
		For all $q\in \mathcal{P}_{n}(K)$, where $K\in\mathcal{K}_{h}$ be an arbitrary element and $n\ge 0$ is a positive integer, there exists a constant $C_{trc}>0$ independent of $h$ such that
		\begin{equation*}
			\norm{q}_{\partial K}\le C_{trc}h_{K}^{-1}\norm{q}_{K}.
		\end{equation*}
	\end{lemma}
	\begin{lemma}[Discrete local inverse inequality, Lemma 1.28 of \cite{di2020hybrid}]\label{inverse}
		For all $q\in \mathcal{P}_{n}(K)$, where $K\in\mathcal{K}_{h}$ be an arbitrary element and $n\ge 0$ is a positive integer, there exists a constant $C_{inv}>0$ independent of $h$ such that
		\begin{equation*}
			\norm{\nabla q}_{\partial K}\le C_{inv}h_{K}^{-1}\norm{q}_{K}.
		\end{equation*} 
	\end{lemma}
	\begin{lemma}[Discrete Poincar\'e inequality, Lemma 2.15 of \cite{di2020hybrid}]\label{Poincare}
		For any $\hat{q}_{h}=(q_{\mathcal{K}},q_{\mathcal{F}})\in V_{h,0}^{k}$ there exists a constant $C_{dp}$ independent of mesh size $h$ such that
		\begin{equation*}
			\norm{q_{\mathcal{K}}}\le C_{dp} \hhonorm{\hat{q}_{h}}.
		\end{equation*}
	\end{lemma}
	
	We now proceed to introduce the notions of elliptic projector. Consider the map $E_{K}:H^{1}(K)\to \mathcal{P}_{k+1}(K)$, where $K\in\mathcal{K}_{h}$ is an arbitrary cell, defined in the following way:
	\begin{eqnarray}\label{ellipticprojection}
		\left\{
		\begin{array}{ll}
			(\nabla(E_{K}^{k+1}(w)-w),\nabla q)=0\,\forall q\in \mathcal{P}_{k+1}(K),\\
			(E_{K}^{k+1}(w)-w,1)_{T}=0.
		\end{array}
		\right.
	\end{eqnarray}
	Define the global elliptic projection as $E_{\mathcal{K}}^{k+1}:H^{1}(\Omega)\to \prod_{K\in\mathcal{K}_{h}}\mathcal{P}_{k}(K)$ by $(E_{\mathcal{K}}^{k+1}(v))|_{K}:=E_{K}^{k+1}(v)$, $\forall v\in H^{1}(\Omega)$.
	\begin{lemma}\label{approx1}
		Let $K\in\mathcal{K}_{h}$ be an arbitrary cell. For any $q\in H^{1}(K)$, we have the following identity
		\begin{equation*}
			R_{K}^{k+1}(\hat{I}_{K}^{k}(q))= E_{K}^{k+1}(q).
		\end{equation*}
	\end{lemma}
	\begin{lemma}[Approximation properties of $L^{2}$ and elliptic projectors, Lemma 1.45 and Lemma 1.48 of \cite{di2020hybrid}]\label{approx2}
		Let $K\in\mathcal{K}_{h}$ be an arbitrary cell. For any $q\in H^{s}(K)$, where $s\in\{0,\ldots,k+1\}$, then for all $m\in\{0,\ldots,s\}$ we have
		\begin{eqnarray*}
			|q-\pi_{K}^{k}(q)|_{H^{m}(K)}\le Ch_{K}^{s-m}|q|_{H^{s}(K)}.
		\end{eqnarray*}
		Further, if we assume $q\in H^{s}(K)$, $s\in\{1,\ldots,k+2\}$, then for all $m\in\{0,\ldots,s\}$ we have
		\begin{eqnarray*}
			|q-E_{K}^{k+1}(q)|_{H^{m}(K)}\le Ch_{K}^{s-m}|q|_{H^{s}(K)}.
		\end{eqnarray*}
	\end{lemma}
	\begin{lemma}[Boundedness of the global projection operator]\label{boundedprojector}
		Let $u\in H^{1}(\Omega)$ be arbitrary. Then we have the following boundedness of the global projection operator.
		\begin{equation*}
			\norm{\hat{I}_{h}^{k}u}_{1,h}\le C\norm{u}_{H^{1}(\Omega)}.
		\end{equation*}
	\end{lemma}
	\begin{proof}
		For any function $u\in H^{1}(K)$, where $K\in\mathcal{K}_{h}$ is arbitrary, from Proposition 2.2 of \cite{di2020hybrid}, we have
		\begin{eqnarray}\label{b1}
			\norm{\hat{I}_{K}^{k}u}_{1,K}^{2}&\le& 	C\norm{u}_{H^{1}(K)}^{2}.
		\end{eqnarray}
		Here $C$ is independent of $h$. It follows from the equivalence of norms (see \eqref{equivalence}) and the from estimate \eqref{b1} that		
		\begin{eqnarray*}
			\norm{\hat{I}_{h}^{k}u}_{1,h}^{2}&\le& C_{eqv}^{-1}\hhonorm{\hat{I}_{h}^{k}u}^{2}\nonumber\\
			&=&C_{eqv}^{-1}\sum_{K\in\mathcal{K}_{h}}\norm{\hat{I}^{k}_{K}u}_{1,K}^{2}\nonumber\\
			&\le&C\sum_{K\in\mathcal{K}_{h}}\norm{u}_{H^{1}(K)}^{2}=C\norm{u}_{H^{1}(\Omega)}^{2}.
		\end{eqnarray*}
		This completes the proof.
	\end{proof}
	\section{\normalsize The continuous-in-time HHO method}\label{sec4}
	In this section, we present the semi-discrete HHO methods for the parabolic integro-differential problem \eqref{model}. Later, we perform the stability analysis and establish some error estimates.
	
	With the previously introduced notations in Section \ref{sec3}, the space semi-discrete problem is given as:
	
	\noindent\rule{\textwidth}{0.4pt}
	\noindent{\bf The semi-discrete HHO method:}
	For any  $\hat{q}_{h}=(q_{\mathcal{K}},q_{\mathcal{F}})\in V_{h,0}^{k}$, seek a hybrid function $\hat{p}_{h}:=(p_{\mathcal{K}},p_{\mathcal{F}})\in C^{1}(0,T;V_{h,0}^{k})$ such that
	\begin{eqnarray}\label{HHOscheme}
		\left\{
		\begin{array}{ll}
			(p_{\mathcal{K}t}(t),q_{\mathcal{K}})+b(\hat{p}_{h}(t),\hat{q}_{h})+\int_{0}^{t}b(\hat{p}_{h}(s),\hat{q}_{h})ds=(f(t),q_\mathcal{K}),\\
			\hat{p}_{h}(0)=\hat{I}_{h}^{k}(g).
		\end{array}
		\right.
	\end{eqnarray}
	\noindent\rule{\textwidth}{0.4pt}
	
	\begin{lemma}[The semi-discrete stability result]
		The semi-discrete HHO scheme \eqref{HHOscheme} satisfies the following stability estimate.
		\begin{equation*}
			\norm{{p}_{\mathcal{K}}(t)}^{2}+\int_{0}^{t}\norm{\hat{p}_{h}(s)}_{1,h}^{2}+\norm{\int_{0}^{t}\hat{p}_{h}(s)\,ds}_{1,h}^{2} \le\norm{g}^{2}+C\int_{0}^{t}\norm{f(s)}^{2}\,ds.
		\end{equation*}
		Here, the constant is given by $C=\frac{C_{dp}}{C_{eqv}}$. 
	\end{lemma}
	\begin{proof}
		Selecting $\hat{q}_{h}=\hat{p}_{h}(t)$ in \eqref{HHOscheme}, we achieve
		\begin{eqnarray*}
			\frac{1}{2}\frac{d}{dt}\norm{{p}_{\mathcal{K}}(t)}^{2}+b(\hat{p}_{h}(t),\hat{p}_{h}(t))+\int_{0}^{t}b(\hat{p}_{h}(s),\hat{p}_{h}(t))\,ds=(f(t),p_\mathcal{K}(t)).
		\end{eqnarray*}
		Taking the integral from $0$ to $t$ $(t>0)$ in the above equation results in,
		\begin{eqnarray}\label{st1}
			&&	\norm{p_{\mathcal{K}}(t)}^{2}+2\int_{0}^{t}b(\hat{p}_{h}(s),\hat{p}_{h}(s))\,ds
			+2\int_{0}^{t}\int_{0}^{s}b(\hat{p}_{h}(\xi),\hat{p}_{h}(s))\,d\xi ds\nonumber\\
			&&=\norm{p_{\mathcal{K}}(0)}^{2}+2\int_{0}^{t}(f(s),p_\mathcal{K}(s))\,ds.
		\end{eqnarray}
		It is not hard to verify that
		\begin{eqnarray}\label{st2}
			2\int_{0}^{t}\int_{0}^{s}b(\hat{p}_{h}(\xi),\hat{p}_{h}(s))\,d\xi ds=\int_{0}^{t}\int_{0}^{t}b(\hat{p}_{h}(\xi),\hat{p}_{h}(s))\,d\xi ds=\norm{\int_{0}^{t}\hat{p}_{h}(s)\,ds}_{1,h}^{2}.~~
		\end{eqnarray}
		Further, from \eqref{st1}-\eqref{st2}, \eqref{bilnorm}, equivalence of $\norm{\cdot}_{1,h}$ and $\norm{\cdot}_{\mathrm{HHO}}$ norms (see \eqref{equivalence}), next by Cauchy-Schwarz and Young's inequality, we obtain
		\begin{eqnarray*}
			&&\norm{p_{\mathcal{K}}(t)}^{2}+2\int_{0}^{t}\norm{\hat{p}_{h}(s)}_{1,h}^{2}+\norm{\int_{0}^{t}\hat{p}_{h}(s)\,ds}_{1,h}^{2}\nonumber\\
			&&\le\norm{\pi_{\mathcal{K}}^{k}(g)}^{2}+\frac{C_{dp}^{2}}{C_{eqv}^{2}}\int_{0}^{t}\norm{f(s)}^{2}\,ds+\frac{C_{eqv}^{2}}{C_{dp}^{2}}\int_{0}^{t}\norm{p_{\mathcal{K}}(s)}^{2}\,ds\nonumber\\
			&&\le\norm{\pi_{\mathcal{K}}^{k}(g)}^{2}+\frac{C_{dp}^{2}}{C_{eqv}^{2}}\int_{0}^{t}\norm{f(s)}^{2}\,ds+C_{eqv}^{2}\int_{0}^{t}\norm{\hat{p}_{h}(s)}_{\mathrm{HHO}}^{2}\,ds\nonumber\\
			&&\le\norm{g}^{2}+C\int_{0}^{t}\norm{f(s)}^{2}\,ds+\int_{0}^{t}\norm{\hat{p}_{h}(s)}_{1,h}^{2}\,ds
		\end{eqnarray*}
		In the last inequality, we have used the stability property of $L^{2}$-projection $\pi_{\mathcal{K}}^{k}(\cdot)$ under the $L^{2}$-norm. Finally, a simple calculation yields the desired stability result.
	\end{proof}
	\subsection{\normalsize Error analysis for the semi-discrete problem}
	We now proceed to determine the error estimate in an energy-type norm. For that, first we split the exact error as $p-\hat{p}_{h}=p-\hat{I}_{h}^{k}(p)+\hat{e}_{h}$, where $\hat{e}_{h}:=\hat{I}_{h}^{k}(p)-\hat{p}_{h}$ is the consistency error and the term $p-\hat{I}_{h}^{k}p$ is the approximation error. We are only concerned about determining estimates for the consistency error $\hat{e}_{h}$ since the approximation error can be dealt with using Lemma \ref{approx2}.
	
	\begin{lemma}[The semi-discrete error equation]
		The continuous-in-time error equation is given by
		\begin{eqnarray}\label{errmain}
			&&(e_{\mathcal{K}t},q_{\mathcal{K}})+b(\hat{e}_{h},\hat{q}_{h})+\int_{0}^{t}b(\hat{e}_{h}(s),\hat{q}_{h})ds\nonumber\\
			&&=\sum_{i=1}^{2}\mathcal{R}_{i}(p,\hat{q}_{h})+s_{h}(\hat{I}^{k}_{h}(p),\hat{q}_{h})+\int_{0}^{t}s_{h}(\hat{I}^{k}_{h}(p(s)),\hat{q}_{h})ds.
		\end{eqnarray}
		Here, 
		\begin{eqnarray*}
			\mathcal{R}_{1}(p,\hat{q}_{h})&=&\sum_{K\in\mathcal{K}_{h}}\sum_{F\in\mathcal{F}_{K}}\langle (\nabla p-\nabla E_{K}^{k+1}(p)) \cdot{\bf n}_{KF},q_{K}|_{F}-q_{F}\rangle_{F},\nonumber\\
			\mathcal{R}_{2}(p,\hat{q}_{h})&=&
			\sum_{K\in\mathcal{K}_{h}}\sum_{F\in\mathcal{F}_{K}}\int_{0}^{t}\langle(\nabla p(s)-\nabla E_{K}^{k+1}(p(s)))\cdot{\bf n}_{KF },q_{K}|_{F}-q_{F}\rangle_{F}ds.
		\end{eqnarray*} 
	\end{lemma}
	\begin{proof}
		The proof begins by multiplying a test function $\hat{q}_{h}=(q_{\mathcal{K}}=(q_{K})_{K\in\mathcal{K}_{h}},q_{\mathcal{F}}=(q_{F})_{F\in\mathcal{F}_{h}})\in V_{h,0}^{k}$ with \eqref{model} and then employing integration by parts to have
		\begin{eqnarray}\label{err1}
			(f,q_{\mathcal{K}})&=&(p_{t},q_{\mathcal{K}})-(\Delta p,q_{\mathcal{K}})-\left(\int_{0}^{t}\Delta p(s)ds,q_{\mathcal{K}}\right)\nonumber\\
			&=&(p_{t},q_{\mathcal{K}})-\sum_{K\in\mathcal{K}_{h}}(\Delta p,q_{K})_{K}-\sum_{K\in\mathcal{K}_{h}}\int_{0}^{t}(\Delta p(s),q_{K})_{K}ds\nonumber\\
			&=&(p_{t},q_{\mathcal{K}})+\sum_{K\in\mathcal{K}_{h}}(\nabla p,\nabla q_{K})_{K}-\sum_{K\in\mathcal{K}_{h}}\sum_{F\in\mathcal{F}_{K}}\langle \nabla p \cdot{\bf n}_{KF},q_{K}|_{F}\rangle_{F}\nonumber\\
			&&+\sum_{K\in\mathcal{K}_{h}}\int_{0}^{t}(\nabla p(s),\nabla q_{K})_{K}ds-\sum_{K\in\mathcal{K}_{h}}\sum_{F\in\mathcal{F}_{K}}\int_{0}^{t}\langle\nabla p(s)\cdot{\bf n}_{KF},q_{K}|_{F}\rangle_{F}ds\nonumber\\
			&=&(p_{t},q_{\mathcal{K}})+\sum_{K\in\mathcal{K}_{h}}(\nabla p,\nabla q_{K})_{K}-\sum_{K\in\mathcal{K}_{h}}\sum_{F\in\mathcal{F}_{K}}\langle \nabla p \cdot{\bf n}_{KF},q_{K}|_{F}-q_{F}\rangle_{F}\nonumber\\
			&&+\sum_{K\in\mathcal{K}_{h}}\int_{0}^{t}(\nabla p(s),\nabla q_{K})_{K}ds-\sum_{K\in\mathcal{K}_{h}}\sum_{F\in\mathcal{F}_{K}}\int_{0}^{t}\langle\nabla p(s)\cdot{\bf n}_{KF},q_{K}|_{F}-q_{F}\rangle_{F}ds.\nonumber\\
		\end{eqnarray}
		In the last equality, we utilize the continuity of flux $\nabla u\cdot {\bf n}_{KF}$ across the interface faces and the fact $\hat{q}_{h}\in V_{h,0}^{k}$, implies $q_{F}=0$ across all the boundary faces, to derive
		\begin{equation*}
			\sum_{K\in\mathcal{K}_{h}}\sum_{F\in\mathcal{F}_{K}}\langle \nabla p \cdot{\bf n}_{KF},q_{F}\rangle_{F}=0,\,\mathrm{and}\,\sum_{K\in\mathcal{K}_{h}}\sum_{F\in\mathcal{F}_{K}}\int_{0}^{t}\langle \nabla p(s) \cdot{\bf n}_{KF},q_{F}\rangle_{F}ds=0.
		\end{equation*}
		Further, applying Lemma \ref{approx1}, later utilizing definition of potential reconstruction operator (see Definition \ref{potentialreconstruction}) and lastly employing the definition of elliptic projection (see \eqref{ellipticprojection})  results in
		\begin{eqnarray*}
			&&(\nabla R_{h}^{k+1}(\hat{I}^{k}_{h}p),\nabla R_{h}^{k+1}(\hat{q}_{h}))
			\nonumber\\
			&&=\sum_{K\in\mathcal{K}_{h}}(\nabla R_{K}^{k+1}(\hat{I}_{K}^{k}p),\nabla R_{K}^{k+1}(\hat{q}_{K}))_{K}\nonumber\\
			&&=\sum_{K\in\mathcal{K}_{h}}(\nabla E_{K}^{k+1}(p),\nabla R_{K}^{k+1}(\hat{q}_{K}))_{K}\nonumber\\
			&&=\sum_{K\in\mathcal{K}_{h}}(\nabla q_{K},\nabla E_{K}^{k+1}(p))_{K}+\sum_{K\in\mathcal{K}_{h}}\sum_{F\in\mathcal{F}_{K}}\langle q_{F}-q_{K}|_{F},{\bf n}_{KF}\cdot\nabla E_{K}^{k+1}(p)\rangle_{\partial K}
			\nonumber\\
			&&=\sum_{K\in\mathcal{K}_{h}}(\nabla q_{K},\nabla p)_{K}+\sum_{K\in\mathcal{K}_{h}}\sum_{F\in\mathcal{F}_{K}}\langle q_{F}-q_{K}|_{F},{\bf n}_{KF}\cdot\nabla E_{K}^{k+1}(p)\rangle_{F}.
		\end{eqnarray*}
		Thus, we have
		\begin{eqnarray}\label{err3}
			\sum_{K\in\mathcal{K}_{h}}(\nabla q_{K},\nabla p)_{K}&=&(\nabla R_{h}^{k+1}(\hat{I}^{k}_{h}p),\nabla R_{h}^{k+1}(\hat{q}_{h}))\nonumber\\
			&&\hspace{-0.1cm}-\sum_{K\in\mathcal{K}_{h}}\sum_{F\in\mathcal{F}_{K}}\langle q_{F}-q_{K}|_{F},{\bf n}_{K}\cdot\nabla E_{K}^{k+1}(p)\rangle_{F}.
		\end{eqnarray}
		Similarly, we can deduce
		\begin{eqnarray}\label{err4}
			\sum_{K\in\mathcal{K}_{h}}\int_{0}^{t}(\nabla q_{K},\nabla p(s))_{K}ds&=&\int_{0}^{t}(\nabla R_{h}^{k+1}(\hat{I}^{k}_{h}p(s)),\nabla R_{h}^{k+1}(\hat{q}_{h}))ds\nonumber\\
			&&\hspace{-0.6cm}-\sum_{K\in\mathcal{K}_{h}}\sum_{F\in\mathcal{F}_{K}}\int_{0}^{t}\langle q_{F}-q_{K}|_{F},{\bf n}_{KF}\cdot\nabla E_{K}^{k+1}(p(s))\rangle_{F}ds.\nonumber\\
		\end{eqnarray}
		Using \eqref{err3}-\eqref{err4} in \eqref{err1}, we derive
		\begin{eqnarray*}
			&&\sum_{K\in\mathcal{K}_{h}}(\pi_{K}^{k}(p_{t}),q_{K})_{K}+(\nabla R_{h}^{k+1}(\hat{I}^{k}_{h}p),\nabla R_{h}^{k+1}(\hat{q}_{h}))\nonumber\\
			&&+\int_{0}^{t}(\nabla R_{h}^{k+1}(\hat{I}^{k}_{h}p(s)),\nabla R_{h}^{k+1}(\hat{q}_{h}))ds\nonumber\\
			&&~~=(f,q_{\mathcal{K}})+\sum_{K\in\mathcal{K}_{h}}\sum_{F\in\mathcal{F}_{K}}\langle (\nabla p-\nabla E_{K}^{k+1}(p)) \cdot{\bf n}_{K},q_{K}|_{F}-q_{F}\rangle_{F}\nonumber\\
			&&~~~~+\sum_{K\in\mathcal{K}_{h}}\sum_{F\in\mathcal{F}_{K}}\int_{0}^{t}\langle(\nabla p(s)-\nabla E_{K}^{k+1}(p(s)))\cdot{\bf n}_{K},q_{K}|_{F}-q_{F}\rangle_{F}ds.
		\end{eqnarray*}
		Adding the stabilization terms $s_{h}(\hat{I}^{k}_{h}(p),\hat{q}_{h})$ and $\int_{0}^{t}s_{h}(\hat{I}^{k}_{h}(p(s)),\hat{q}_{h})ds$ to the above equation, we achieve 
		\begin{eqnarray}\label{err5}
			&&(\pi_{\mathcal{K}}^{k}(p_{t}),q_{\mathcal{K}})+b(\hat{I}_{h}^{k}(p),\hat{q}_{h})+\int_{0}^{t}b(\hat{I}_{h}^{k}(p(s)),\hat{q}_{h})ds\nonumber\\
			&&=(f,q_{\mathcal{K}})+\sum_{i=1}^{2}\mathcal{R}_{i}(p,\hat{q}_{h})+s_{h}(\hat{I}^{k}_{h}(p),\hat{q}_{h})+\int_{0}^{t}s_{h}(\hat{I}^{k}_{h}(p(s)),\hat{q}_{h})ds.~~
		\end{eqnarray}
		The desired error equation \eqref{errmain} is obtained by subtracting the first equation of \eqref{HHOscheme} from \eqref{err5}.
	\end{proof}
	\begin{lemma}\label{bounds}
		Under the assumption that the true solution of \eqref{model} has the regularity $p\in L^{\infty}(0,T;H^{m+1}(\Omega))$, $m\in\{0,1,2,\ldots k\}$, we have the following bounds of the residuals appearing in the semi-discrete error equation given as
		\begin{eqnarray*}
			|\mathcal{R}_{1}(p,\hat{q}_{h})|&\le& Ch^{k+1}\norm{p}_{H^{k+2}(\Omega)}\norm{\hat{q}_{h}}_{1,h},\\
			|\mathcal{R}_{2}(p,\hat{q}_{h})|&\le& Ch^{k+1}\norm{\int_{0}^{t}p(s)ds}_{H^{k+2}(\Omega)}\norm{\hat{q}_{h}}_{1,h},\\
			|s_{h}(\hat{I}^{k}_{h}(p),\hat{q}_{h})|&\le&Ch^{k+1}\norm{p}_{H^{k+2}(\Omega)}\norm{\hat{q}_{h}}_{1,h},\\
			\left|\int_{0}^{t}s_{h}(\hat{I}^{k}_{h}(p(s)),\hat{q}_{h})ds\right|&\le&Ch^{k+1}\norm{\int_{0}^{t}p(s)ds}_{H^{k+2}(\Omega)}\norm{\hat{q}_{h}}_{1,h}.
		\end{eqnarray*}
	\end{lemma}
	\begin{proof}
		Invoking Cauchy-Schwarz inequality, trace inequality (see Lemma \ref{trace}) and approximation properties (cf. Lemma \ref{approx2}) of elliptic projector, we obtain the desired estimates for residuals $\mathcal{R}_{i}$ $i=\{1,2\}$. Further, the estimates for the stabilization terms can be derived as in inequality (2.31) of \cite{di2020hybrid}.
	\end{proof}
	\begin{thm}\label{sdH1}
		Assume that $p\in L^{\infty}(0,T;H^{k+2}(\Omega)\cap H_{0}^{1}(\Omega))$ be solution of the problem \eqref{model}. Then we have the following discrete-energy norm error estimate for the consistency error.
		\begin{eqnarray*}
			\norm{\hat{e}_{h}(t)}_{1,h}^{2}
			&\le& C(1+T)h^{2(k+1)}\norm{p}_{L^{\infty}(0,T;H^{k+2}(\Omega))}^{2}\nonumber\\
			&&+Ch^{2(k+1)}\norm{p}_{L^{2}0,T;(H^{k+2}(\Omega))}^{2}.\nonumber\\
		\end{eqnarray*}
		Additionally the constant $C$ depends on the final time $T$.
	\end{thm}
	\begin{proof}
		We begin by replacing $\hat{q}_{h}=\hat{e}_{ht}(t)$ in the error equation \eqref{errmain} to get
		\begin{eqnarray*}
			&&\norm{e_{\mathcal{K}t}}^{2}+\frac{1}{2}\frac{d}{dt}\norm{\hat{e}_{h}(t)}_{1,h}^{2}+\int_{0}^{t}b(\hat{e}_{h}(s),\hat{e}_{ht}(t))ds\nonumber\\
			&&=\sum_{i=1}^{2}\mathcal{R}_{i}(p,\hat{e}_{ht})+s_{h}(\hat{I}^{k}_{h}(p),\hat{e}_{ht})+\int_{0}^{t}s_{h}(\hat{I}^{k}_{h}(p(s)),\hat{e}_{ht})ds.
		\end{eqnarray*}
		Taking the integral from $0$ to $t$ on both sides of above equation, we achieve
		\begin{eqnarray}\label{est1}
			&&2\int_{0}^{t}\norm{e_{\mathcal{K}\xi}(\xi)}^{2}+\norm{\hat{e}_{h}(t)}_{1,h}^{2}\nonumber\\
			&&=\norm{\hat{e}_{h}(0)}_{1,h}^{2}+2\sum_{i=1}^{2}\int_{0}^{t}\mathcal{R}_{i}(p(\xi),\hat{e}_{h\xi}(\xi))d\xi+2\int_{0}^{t}s_{h}(\hat{I}^{k}_{h}(p(\xi)),\hat{e}_{ht}(\xi))d\xi\nonumber\\
			&&~~+2\int_{0}^{t}\int_{0}^{\xi}s_{h}(\hat{I}^{k}_{h}(p(s)),\hat{e}_{h\xi}(\xi))dsd\xi-2\int_{0}^{t}\int_{0}^{\xi}b(\hat{e}_{h}(s),\hat{e}_{h\xi}(\xi))dsd\xi.~~~~~~
		\end{eqnarray}
		Note that $\hat{e}_{h}(0)=\hat{I}_{h}^{k}(p(0))-\hat{p}_{h}(0)=0$ and using integration by parts, we derive
		\begin{eqnarray}\label{est2}
			2\int_{0}^{t}\int_{0}^{\xi}b(\hat{e}_{h}(s),\hat{e}_{h\xi}(\xi))dsd\xi
			&=&2\int_{0}^{t}b\left(\int_{0}^{\xi}\hat{e}_{h}(s)ds,\hat{e}_{h\xi}(\xi)\right)d\xi\nonumber\\
			&=&2b\left(\int_{0}^{t}\hat{e}_{h}(s)ds,\hat{e}_{h}(t)\right)\nonumber\\
			&&-2\int_{0}^{t}b(\hat{e}_{h}(\xi),\hat{e}_{h}(\xi))d\xi
			\nonumber\\
			&=&2\int_{0}^{t}b(\hat{e}_{h}(s),\hat{e}_{h}(t))ds-2\int_{0}^{t}\norm{\hat{e}_{h}(\xi)}_{1,h}^{2}d\xi.\nonumber\\
		\end{eqnarray}
		In similar pattern as above, we derive
		\begin{eqnarray}\label{est2*}
			2\int_{0}^{t}\int_{0}^{\xi}s_{h}(\hat{I}^{k}_{h}(p(s)),\hat{e}_{h\xi}(\xi))dsd\xi&=&2\int_{0}^{t}s_{h}(\hat{I}_{h}^{k}(p(s)),\hat{e}_{h}(t))ds\nonumber\\&&-2\int_{0}^{t}s_{h}(\hat{I}_{h}^{k}(p(\xi)),\hat{e}_{h}(\xi))d\xi.
		\end{eqnarray}
		Utilizing \eqref{est2}-\eqref{est2*} in \eqref{est1}, we deduce
		\begin{eqnarray}\label{est3}
			&&2\int_{0}^{t}\norm{e_{\mathcal{K}\xi}(\xi)}^{2}+\norm{\hat{e}_{h}(t)}_{1,h}^{2}\nonumber\\
			&&=2\sum_{i=1}^{2}\int_{0}^{t}\mathcal{R}_{i}(p(\xi),\hat{e}_{h\xi}(\xi))d\xi+2\int_{0}^{t}s_{h}(\hat{I}^{k}_{h}(p(\xi)),\hat{e}_{ht}(\xi))d\xi\nonumber\\
			&&~~+2\int_{0}^{t}s_{h}(\hat{I}_{h}^{k}(p(s)),\hat{e}_{h}(t))ds-2\int_{0}^{t}s_{h}(\hat{I}_{h}^{k}(p(\xi)),\hat{e}_{h}(\xi))d\xi\nonumber\\
			&&~~-2\int_{0}^{t}b(\hat{e}_{h}(s),\hat{e}_{h}(t))ds+2\int_{0}^{t}\norm{\hat{e}(\xi)}_{1,h}^{2}d\xi.
		\end{eqnarray}
		We now proceed to find a bound for each of the terms in right hand side of \eqref{est3}.
		
		From integration by parts, Lemma \ref{bounds} and utilizing Young's inequality, for $i=\{1,2\}$ it follows that
		\begin{eqnarray}\label{est4}
			\int_{0}^{t}\mathcal{R}_{i}(p(\xi),\hat{e}_{h\xi}(\xi))d\xi&=&\mathcal{R}_{i}(p(t),\hat{e}_{h}(t))-\int_{0}^{t}\mathcal{R}_{i}(p_{\xi}(\xi),\hat{e}_{h}(\xi))d\xi\nonumber\\
			&\le&Ch^{k+1}\norm{p(t)}_{k+2}\norm{\hat{e}_{h}(t)}_{1,h}\nonumber\\
			&&+Ch^{k+1}\int_{0}^{t}\norm{p(\xi)}_{k+2}\norm{\hat{e}_{h}(\xi)}_{1,h}d\xi\nonumber\\
			&\le&Ch^{2(k+1)}\norm{p(t)}_{H^{k+2}(\Omega)}^{2}+\frac{1}{20}\norm{\hat{e}_{h}(t)}_{1,h}^{2}\nonumber\\
			&&+Ch^{2(k+1)}\int_{0}^{t}\norm{p(\xi)}_{H^{k+2}(\Omega)}^{2}d\xi+\frac{1}{20}\int_{0}^{t}\norm{\hat{e}_{h}(\xi)}_{1,h}^{2}d\xi.\nonumber\\
		\end{eqnarray}
		Arguing as above, we determine
		\begin{eqnarray}\label{est5}
			\int_{0}^{t}s_{h}(\hat{I}^{k}_{h}(p(\xi)),\hat{e}_{ht}(\xi))d\xi&\le&Ch^{2(k+1)}\norm{p(t)}_{H^{k+2}(\Omega)}^{2}+\frac{1}{20}\norm{\hat{e}_{h}(t)}_{1,h}^{2}\nonumber\\
			&&+Ch^{2(k+1)}\int_{0}^{t}\norm{p(\xi)}_{H^{k+2}(\Omega)}^{2}d\xi+\frac{1}{20}\int_{0}^{t}\norm{\hat{e}_{h}(\xi)}_{1,h}^{2}d\xi.\nonumber\\
		\end{eqnarray}
		Further, from Lemma \ref{bounds} and Young's inequality, we derive
		\begin{eqnarray}\label{est6}
			\int_{0}^{t}s_{h}(\hat{I}_{h}^{k}(p(s)),\hat{e}_{h}(t))ds&\le&Ch^{k+1}\norm{\int_{0}^{t}p(s)ds}_{H^{k+2}(\Omega)}\norm{\hat{e}_{h}(t)}_{1,h}\nonumber\\
			&\le&Ch^{2(k+1)}\norm{\int_{0}^{t}p(s)ds}_{H^{k+2}(\Omega)}^{2}+\frac{1}{20}\norm{\hat{e}_{h}(t)}_{1,h}^{2}\nonumber\\
			&\le&CTh^{2(k+1)}\norm{p}_{L^{\infty}(0,T;H^{k+2}(\Omega))}^{2}+\frac{1}{20}\norm{\hat{e}_{h}(t)}_{1,h}^{2}.~~~~~~~~
		\end{eqnarray}
		In the same way as above, we obtain
		\begin{eqnarray}\label{est7}
			\int_{0}^{t}s_{h}(\hat{I}_{h}^{k}(p(\xi)),\hat{e}_{h}(\xi))d\xi&\le&Ch^{k+1}\int_{0}^{t}\norm{p(\xi)}_{H^{k+2}(\Omega)}\norm{\hat{e}_{h}(\xi)}_{1,h}\nonumber\\
			&\le&Ch^{2(k+1)}\int_{0}^{t}\norm{p(\xi)}_{H^{k+2}(\Omega)}^{2}d\xi+\frac{1}{20}\int_{0}^{t}\norm{\hat{e}_{h}(\xi)}_{1,h}^{2}d\xi.\nonumber\\
		\end{eqnarray}
		Due to boundedness property of bilinear form (see \eqref{boundedness}), then applying Cauchy-Schwarz and Young's inequalities, we have
		\begin{eqnarray}\label{est8}
			\int_{0}^{t}b(\hat{e}_{h}(s),\hat{e}_{h}(t))ds&=&b\left(\int_{0}^{t}\hat{e}_{h}(s)ds,\hat{e}_{h}(t)\right)\nonumber\\
			&&\le \norm{\int_{0}^{t}\hat{e}_{h}(s)ds}_{1,h}\norm{\hat{e}_{h}(t)}_{1,h}\nonumber\\
			&&\le \sqrt{T}\left(\int_{0}^{t}\norm{\hat{e}_{h}(s)}_{1,h}^{2}ds\right)^{\frac{1}{2}}\norm{\hat{e}_{h}(t)}_{1,h}\nonumber\\
			&&\le CT\int_{0}^{t}\norm{\hat{e}_{h}(s)}_{1,h}^{2}ds+\frac{1}{20}\norm{\hat{e}_{h}(t)}_{1,h}^{2}.
		\end{eqnarray}
		Combining the results \eqref{est4}-\eqref{est8} and utilizing it in \eqref{est3}, we determine
		\begin{eqnarray*}
		\norm{\hat{e}_{h}(t)}_{1,h}^{2}
			&\le& C(1+T)h^{2(k+1)}\norm{p}_{L^{\infty}(0,T;H^{k+2}(\Omega))}^{2}\nonumber\\
			&&+Ch^{2(k+1)}\norm{p}_{L^{2}(0,T;H^{k+2}(\Omega))}^{2}+C\int_{0}^{t}\norm{\hat{e}_{h}(s)}_{1,h}^{2}ds.
		\end{eqnarray*}
		The desired error estimate is then obtained from the above inequality by simply employing the continuous Gronwall's Lemma.
	\end{proof}	
		\begin{thm}
		Assume that $p\in L^{\infty}(0,T;H^{k+2}(\Omega)\cap H_{0}^{1}(\Omega))$ be solution of the problem \eqref{model}. Then we have the following discrete-energy norm estimate for the actual error.
		\begin{eqnarray*}
			\norm{\nabla p(t)-\nabla R_{h}^{k+1}(\hat{p}_{h}(t))}^{2}
			&\le&  C(1+T)h^{2(k+1)}\norm{p}_{L^{\infty}(0,TH^{k+2}(\Omega))}^{2}\nonumber\\
			&&+Ch^{2(k+1)}\norm{p}_{L^{2}(0,T;H^{k+2}(\Omega))}^{2}.\nonumber
		\end{eqnarray*}
		Additionally, the constant $C$ depends on the final time $T$.
	\end{thm}
	\begin{proof}
		From triangle inequality, Lemma \ref{approx1}, approximation property of the elliptic projection (see Lemma \ref{approx2}) and Lemma \ref{sdH1}, we derive
		\begin{eqnarray*}
			&&\norm{\nabla p(t)-\nabla R_{h}^{k+1}(\hat{p}_{h}(t))}^{2}\nonumber\\
			&&=\sum_{K\in\mathcal{K}_{h}}\norm{\nabla p(t)-\nabla R_{K}^{k+1}(\hat{p}_{K}(t))}_{K}^{2}\nonumber\\
			&&\le C\sum_{K\in\mathcal{K}_{h}}\norm{\nabla p(t)-\nabla R_{K}^{k+1}(\hat{I}_{K}^{k}(p(t)))}_{K}^{2}\nonumber\\
			&&~~+C\sum_{K\in\mathcal{K}_{h}}\norm{\nabla R_{K}^{k+1}(\hat{I}_{K}^{k}(p(t)))-\nabla R_{K}^{k+1}(\hat{p}_{K}(t))}_{K}^{2}\nonumber\\
			&&\le C\sum_{K\in\mathcal{K}_{h}}\norm{\nabla p(t)-\nabla E_{K}^{k+1}(p(t))}_{K}^{2}+C\sum_{K\in\mathcal{K}_{h}}\norm{\nabla R_{K}^{k+1}(\hat{I}_{K}^{k}(p(t))-\hat{p}_{K}(t))}_{K}^{2}\nonumber\\
			&&\le C\sum_{K\in\mathcal{K}_{h}}\norm{\nabla p(t)-\nabla E_{K}^{k+1}(p(t))}^{2}+C\norm{\hat{e}_{h}(t)}_{1,h}^{2}\nonumber\\
			&&\le C(1+T)h^{2(k+1)}\norm{p}_{L^{\infty}(0,TH^{k+2}(\Omega))}^{2}+Ch^{2(k+1)}\norm{p}_{L^{2}(0,T;H^{k+2}(\Omega))}^{2}.~~~~~
		\end{eqnarray*}
		This completes the proof.
	\end{proof}
	\section{\normalsize The complete discrete HHO method}\label{sec5}
	This section is dedicated to the analysis of stability and convergence of the complete discrete HHO scheme for the linear parabolic integro-differential problem \eqref{model} while employing a Crank-Nicolson time discretization.
	
	The time domain $[0,T]$ is uniformly partitioned into $M$ number of sub-intervals $J_{i}=[t_{i-1},t_{i}]$, $i\in\{0,\ldots,M\}$ with $t_0=0$, $t_{i}=t_{0}+i\tau$, where $\tau=T/M$ and $t_{M}=T$. Again, assume $\phi:[0,T]\to L^{2}(\Omega)$ to be a continuous function in time then we denote $\phi^{n}:=\phi({\bf x},t_{n})$.
	Then let $\{\theta\}_{n=0}^{N}\in L^{2}(\Omega)$, we define the difference quotients in the following manner:
	\begin{eqnarray*}
		\partial_{\tau}\theta^{n}=\frac{\theta^{n+1}-\theta^{n}}{\tau},\quad\theta^{n+\frac{1}{2}}=\frac{\theta^{n+1}+\theta^{n}}{2}.
	\end{eqnarray*}
	\noindent\rule{\textwidth}{0.4pt}
	\noindent{\bf The  complete discrete HHO method:}
	For any  $\hat{q}_{h}=(q_{\mathcal{K}},q_{\mathcal{F}})\in V_{h,0}^{k}$, seek a hybrid function $\hat{P}_{h}^{n+1}:=(P_{\mathcal{K}}^{n+1},P_{\mathcal{F}}^{n+1})\in V_{h,0}^{k}$, $n\in\{0,1,2,\ldots,M-1\}$, such that
	\begin{eqnarray}\label{FDHHOscheme}
		\left\{
		\begin{array}{ll}
			(\partial_{\tau} P_{\mathcal{K}}^{n},q_{\mathcal{K}})+b(\hat{P}_{h}^{n+\frac{1}{2}},\hat{q}_{h})+\mathcal{I}^{n+\frac{1}{2}}(\hat{q}_{h})=(f^{n+\frac{1}{2}},q_\mathcal{K}),\\
			\hat{P}_{h}^{0}=\hat{I}_{h}^{k}(g).
		\end{array}
		\right.
	\end{eqnarray}
	Here, \begin{equation}\label{sum}
		\mathcal{I}^{n}(\hat{q}_{h})=\frac{\tau}{2}\sum_{j=0}^{n}w_{j}b(\hat{P}_{h}^{j},\hat{q}_{h}),
	\end{equation}
	with the corresponding weights defined as follows:
	\begin{eqnarray*}
		w_{j}=\left\{
		\begin{array}{ll}
			1 ,& j=\{0,n\},\\
			2 ,& j=\{2,3,\ldots,n-1\}.
		\end{array}
		\right.
	\end{eqnarray*}
	\noindent\rule{\textwidth}{0.4pt}
	\begin{rem}\label{rewrite}
		We can write the integral term $\mathcal{I}^{n+\frac{1}{2}}(\hat{q}_{h})$ in the following way: 
		\begin{equation*}
			\mathcal{I}^{n+\frac{1}{2}}(\hat{q}_{h})=\tau\sum_{j=0}^{n-1}b(\hat{P}_{h}^{j+\frac{1}{2}},\hat{q}_{h})+\frac{\tau}{2}b(\hat{P}_{h}^{n+\frac{1}{2}},\hat{q}_{h}).
		\end{equation*}
		For the case $n=0$, from \eqref{sum} we directly have
		\begin{equation*}
			\mathcal{I}^{\frac{1}{2}}(\hat{q}_{h})=\frac{\tau}{2}b(\hat{P}_{h}^{\frac{1}{2}},\hat{q}_{h}).
		\end{equation*}
		This necessitates that $\sum_{j=0}^{n-1}b(\hat{p}_{h}^{n+\frac{1}{2}},\hat{q}_{h})=0$ for when $n=0$. We use this fact in later analysis for simplicity of presentation. 
	\end{rem}
	
	\begin{lemma}[The complete discrete stability result]
		The  complete discrete HHO scheme \eqref{HHOscheme} satisfies the following stability estimate.
		\begin{eqnarray*}
			\norm{P_{\mathcal{K}}^{l+1}}^{2}+\tau\sum_{n=0}^{l}\norm{\hat{P}_{h}^{n+\frac{1}{2}}}_{1,h}^{2}
			\le C\left(\norm{g}^{2}+\frac{\tau}{2}\sum_{n=0}^{l+1}\norm{f^{n+\frac{1}{2}}}^{2}\right).
		\end{eqnarray*}
		Here, $l\in\{1,2,\ldots,M-1\}$.
	\end{lemma}
	\begin{proof}
		We start by setting $\hat{q}_{h}=\hat{P}_{h}^{n+\frac{1}{2}}$ in the first equation of \eqref{FDHHOscheme} to obtain
		\begin{eqnarray}\label{fdst1}
			(\partial_{\tau} P_{\mathcal{K}}^{n},P_{\mathcal{K}}^{n+\frac{1}{2}})+b(\hat{P}_{h}^{n+\frac{1}{2}},\hat{P}_{h}^{n+\frac{1}{2}})+\mathcal{I}^{n+\frac{1}{2}}(\hat{P}_{h}^{n+\frac{1}{2}})=(f^{n+\frac{1}{2}},P_{\mathcal{K}}^{n+\frac{1}{2}}).
		\end{eqnarray}
		It is easy to verify that
		\begin{eqnarray}\label{fdst2}
			(\partial_{\tau} P_{\mathcal{K}}^{n},P_{\mathcal{K}}^{n+\frac{1}{2}})=\frac{1}{2\tau}\left(\norm{P_{\mathcal{K}}^{n+1}}^{2}-\norm{P_{\mathcal{K}}^{n}}^{2}\right).
		\end{eqnarray}
		From \eqref{fdst1}-\eqref{fdst2}, we get
		\begin{eqnarray}\label{fdst3}
			&&\norm{P_{\mathcal{K}}^{n+1}}^{2}-\norm{P_{\mathcal{K}}^{n}}^{2}+2\tau\norm{\hat{P}_{h}^{n+\frac{1}{2}}}_{1,h}^{2}+	2\tau^{2}\sum_{j=0}^{n-1}b(\hat{P}_{h}^{j+\frac{1}{2}},\hat{P}_{h}^{n+\frac{1}{2}})+\tau^{2} b(\hat{P}_{h}^{n+\frac{1}{2}},\hat{P}_{h}^{n+\frac{1}{2}})\nonumber\\
			&&=2\tau(f^{n+\frac{1}{2}},P_{\mathcal{K}}^{n+\frac{1}{2}}).
		\end{eqnarray}
		Simplifying the above equation, we achieve
		\begin{eqnarray}\label{fdst4}
			&&\norm{P_{\mathcal{K}}^{n+1}}^{2}-\norm{P_{\mathcal{K}}^{n}}^{2}+(2\tau+\tau^{2})\norm{\hat{P}_{h}^{n+\frac{1}{2}}}_{1,h}^{2}\nonumber\\
			&&=2\tau(f^{n+\frac{1}{2}},P_{\mathcal{K}}^{n+\frac{1}{2}})-2\tau^{2}\sum_{j=0}^{n-1}b(\hat{P}_{h}^{j+\frac{1}{2}},\hat{p}_{h}^{n+\frac{1}{2}})\nonumber\\
			&&\le 	2\tau|(f^{n+\frac{1}{2}},P_{\mathcal{K}}^{n+\frac{1}{2}})|+	2\tau^{2}\left|\sum_{j=0}^{n-1}b(\hat{P}_{h}^{j+\frac{1}{2}},\hat{P}_{h}^{n+\frac{1}{2}})\right|.
		\end{eqnarray}
		We now determine bounds for the terms on the right hand side of \eqref{fdst4}.
		
		It follows from boundedness prorperty (see \eqref{boundedness}), Young's inequality and Cauchy-Schwarz inequality that
		\begin{eqnarray}\label{fdst5}
			2\tau^{2}\left|\sum_{j=0}^{n-1}b(\hat{P}_{h}^{j+\frac{1}{2}},\hat{P}_{h}^{n+\frac{1}{2}})\right|&\le&2\tau^{2}\norm{\hat{P}_{h}^{n+\frac{1}{2}}}_{1,h}\sum_{j=0}^{n-1}\norm{\hat{P}_{h}^{j+\frac{1}{2}}}_{1,h}\nonumber\\
			&\le&\tau\norm{\hat{P}_{h}^{n+\frac{1}{2}}}_{1,h}^{2}+\tau^{3}\left(\sum_{j=0}^{n-1}\norm{\hat{P}_{h}^{j+\frac{1}{2}}}_{1,h}\right)^{2}\nonumber\\
			&\le&\tau\norm{\hat{P}_{h}^{n+\frac{1}{2}}}_{1,h}^{2}+n\tau^{3}\sum_{j=0}^{n-1}\norm{\hat{P}_{h}^{j+\frac{1}{2}}}_{1,h}^{2}\nonumber\\
			&\le&\tau\norm{\hat{P}_{h}^{n+\frac{1}{2}}}_{1,h}^{2}+T\tau^{2}\sum_{j=0}^{n-1}\norm{\hat{P}_{h}^{j+\frac{1}{2}}}_{1,h}^{2}.
		\end{eqnarray}
		Again by Cauchy-Schwarz inequality and Young's inequality, we derive
		\begin{eqnarray}\label{fdst6}
			2\tau|(f^{n+\frac{1}{2}},P_{\mathcal{K}}^{n+\frac{1}{2}})|&\le&2\tau\norm{f^{n+\frac{1}{2}}}\norm{P_{\mathcal{K}}^{n+\frac{1}{2}}}\nonumber\\
			&\le&\tau\norm{f^{n+\frac{1}{2}}}\left(\norm{P_{\mathcal{K}}^{n+1}}+\norm{P_{\mathcal{K}}^{n}}\right)
			\nonumber\\
			&\le&\frac{\tau}{2}\norm{f^{n+\frac{1}{2}}}^{2}+\tau(\norm{P_{\mathcal{K}}^{n+1}}^{2}+\norm{P_{\mathcal{K}}^{n}}^{2}).
		\end{eqnarray}
		Combining \eqref{fdst4}-\eqref{fdst6}, we determine
		\begin{eqnarray}\label{fdst7}
			&&\norm{P_{\mathcal{K}}^{n+1}}^{2}-\norm{P_{\mathcal{K}}^{n}}^{2}+(\tau+\tau^{2})\norm{\hat{P}_{h}^{n+\frac{1}{2}}}_{1,h}^{2}\nonumber\\
			&&\le \frac{\tau}{2}\norm{f^{n+\frac{1}{2}}}^{2}+T\tau^{2}\sum_{j=0}^{n-1}\norm{\hat{P}_{h}^{j+\frac{1}{2}}}_{1,h}^{2}+\tau(\norm{P_{\mathcal{K}}^{n+1}}^{2}+\norm{P_{\mathcal{K}}^{n}}^{2}).
		\end{eqnarray}
		Summing from $n=0$ to $l$, $l\in\{0,1,2,3,\ldots,M-1\}$, we deduce
		\begin{eqnarray}\label{fdst8}
			&&\norm{P_{\mathcal{K}}^{l+1}}^{2}+\tau\sum_{n=0}^{l}\norm{\hat{P}_{h}^{n+\frac{1}{2}}}_{1,h}^{2}\nonumber\\
			&&\le \norm{P_{\mathcal{K}}^{0}}^{2}+ \frac{\tau}{2}\sum_{n=0}^{l}\norm{f^{n+\frac{1}{2}}}^{2}+C\tau^{2}\sum_{n=0}^{l}\sum_{j=0}^{n-1}\norm{\hat{P}_{h}^{j+\frac{1}{2}}}_{1,h}^{2}+2\tau\sum_{n=0}^{l+1}\norm{P_{\mathcal{K}}^{n}}^{2}.\nonumber\\
			&&\le \norm{g}^{2}+ \frac{\tau}{2}\sum_{n=0}^{l+1}\norm{f^{n+\frac{1}{2}}}^{2}+C\tau\sum_{n=0}^{l+1}\norm{P_{\mathcal{K}}^{n+1}}^{2}+C\tau^{2}\sum_{n=0}^{l+1}\sum_{j=0}^{n}\norm{\hat{P}_{h}^{j+\frac{1}{2}}}_{1,h}^{2}.\nonumber\\
		\end{eqnarray}
		In the last inequality we have used the stability of the projection $\norm{P_{\mathcal{K}}^{0}}=\norm{\pi_{\mathcal{K}}(g)}\le \norm{g}$. The required stability estimate now follows from \eqref{fdst8} by simply applying the discrete Gronwall lemma.
	\end{proof}
	\subsection{\normalsize Error analysis for the complete discrete problem}
	At time $t=t_{n}$, the complete discrete error can be split as
	$$p^{n}-\hat{P}_{h}^{n}=p^{n}-I_{h}^{k}(p^{n})+I_{h}^{k}(p^{n})-\hat{P}_{h}^{n}.$$
	Define by $\hat{\mathcal{E}}_{h}^{n}:=I_{h}^{k}(p^{n})-\hat{P}_{h}^{n}$ as the complete discrete consistency error. Hence we have $\hat{\mathcal{E}}_{h}^{n}=\{\mathcal{E}_{\mathcal{K}}^{n}=\pi_{\mathcal{K}}^{k}(p^{n})-P_{\mathcal{K}}^{n},\mathcal{E}_{\mathcal{F}}^{n}=\pi_{\mathcal{F}}^{k}(p^{n})-P_{\mathcal{F}}^{n}\}\in V_{h,0}^{k}$.
	\begin{lemma}
		The complete discrete consistency error satisfies the following error equation.
		\begin{eqnarray}\label{erreqn}
			(\partial_{\tau} \mathcal{E}_{\mathcal{K}}^{n},q_{\mathcal{K}})+b(\hat{\mathcal{E}}_{h}^{n+\frac{1}{2}},\hat{q}_{h})&=&\sum_{i=1}^{2}\mathcal{R}_{i}(p^{n+\frac{1}{2}},\hat{q}_{h})-(p_{t}^{n+\frac{1}{2}}-\partial_{\tau}p^{n},q_{\mathcal{K}})\nonumber\\
			&&+\mathcal{I}^{n+\frac{1}{2}}(\hat{q}_{h})-\int_{0}^{t_{n+\frac{1}{2}}}b(\hat{I}_{h}^{k}(p(s)),\hat{q}_{h})ds\nonumber\\
			&&+s_{h}(\hat{I}^{k}_{h}(p^{n+\frac{1}{2}}),\hat{q}_{h})+\int_{0}^{t_{n+\frac{1}{2}}}s_{h}(\hat{I}^{k}_{h}(p(s)),\hat{q}_{h})ds.\nonumber\\
		\end{eqnarray}
	\end{lemma}
	\begin{proof}
		Consider \eqref{err5} at time $t=t_{n+\frac{1}{2}}$ given by
		\begin{eqnarray*}
			&&(\pi_{\mathcal{K}}^{k}(p_{t}^{n+\frac{1}{2}}),q_{\mathcal{K}})+b(\hat{I}_{h}^{k}(p^{n+\frac{1}{2}}),\hat{q}_{h})+\int_{0}^{t_{n+\frac{1}{2}}}b(\hat{I}_{h}^{k}(p(s)),\hat{q}_{h})ds\nonumber\\
			&&=(f^{n+\frac{1}{2}},q_{\mathcal{K}})+\sum_{i=1}^{2}\mathcal{R}_{i}(p^{n+\frac{1}{2}},\hat{q}_{h})+s_{h}(\hat{I}^{k}_{h}(p^{n+\frac{1}{2}}),\hat{q}_{h})+\int_{0}^{t_{n+\frac{1}{2}}}s_{h}(\hat{I}^{k}_{h}(p(s)),\hat{q}_{h})ds.\nonumber\\
		\end{eqnarray*}
		Rearranging the above equation, we get
		\begin{eqnarray*}
			&&(\partial_{\tau}p^{n},q_{\mathcal{K}})+b(\hat{I}_{h}^{k}(p^{n+\frac{1}{2}}),\hat{q}_{h})+\int_{0}^{t_{n+\frac{1}{2}}}b(\hat{I}_{h}^{k}(p(s)),\hat{q}_{h})ds\nonumber\\
			&&=(f^{n+\frac{1}{2}},q_{\mathcal{K}})+\sum_{i=1}^{2}\mathcal{R}_{i}(p^{n+\frac{1}{2}},\hat{q}_{h})+(\partial_{\tau}p^{n}-\pi_{\mathcal{K}}^{k}(p_{t}^{n+\frac{1}{2}}),q_{\mathcal{K}})\nonumber\\
			&&~~+s_{h}(\hat{I}^{k}_{h}(p^{n+\frac{1}{2}}),\hat{q}_{h})+\int_{0}^{t_{n+\frac{1}{2}}}s_{h}(\hat{I}^{k}_{h}(p(s)),\hat{q}_{h})ds.
		\end{eqnarray*}
		Applying the definition of $L^{2}$-projection leads to
		\begin{eqnarray}\label{fderr1}
			&&(\partial_{\tau}\pi_{\mathcal{K}}^{k}(p^{n}),q_{\mathcal{K}})+b(\hat{I}_{h}^{k}(p^{n+\frac{1}{2}}),\hat{q}_{h})+\int_{0}^{t_{n+\frac{1}{2}}}b(\hat{I}_{h}^{k}(p(s)),\hat{q}_{h})ds\nonumber\\
			&&=(f^{n+\frac{1}{2}},q_{\mathcal{K}})+\sum_{i=1}^{2}\mathcal{R}_{i}(p^{n+\frac{1}{2}},\hat{q}_{h})+(\partial_{\tau}p^{n}-p_{t}^{n+\frac{1}{2}},q_{\mathcal{K}})\nonumber\\
			&&~~+s_{h}(\hat{I}^{k}_{h}(p^{n+\frac{1}{2}}),\hat{q}_{h})+\int_{0}^{t_{n+\frac{1}{2}}}s_{h}(\hat{I}^{k}_{h}(p(s)),\hat{q}_{h})ds.
		\end{eqnarray}
		Subtracting \eqref{FDHHOscheme} from \eqref{fderr1}, achieve the desired consistency error equation.
	\end{proof}
	\begin{thm}
		Let the solution of model \eqref{model} satisfy
		\begin{eqnarray*}
			&&p\in  H^{1}(0,T;H^{k+2}(\Omega)),\,p_{t}\in  H^{1}(0,T;H^{1}(\Omega)),\\
			&&\,p_{tt}\in  L^{1}(0,T;H^{1}(\Omega)),\,p_{ttt}\in  L^{2}(0;T;L^{2}(\Omega)).
		\end{eqnarray*} 
		Then for the complete discrete consistency error with $l\in\{1,2,\ldots,M-1\}$, we have the following estimate.
		\begin{eqnarray}\label{errorest1}
			\norm{\mathcal{E}_{\mathcal{K}}^{l+1}}^{2}+\tau\sum_{n=0}^{l}\norm{\hat{\mathcal{E}}_{h}^{n+\frac{1}{2}}}_{1,h}^{2}
			&\le& Ch^{2k+2}\norm{p}_{L^{\infty}(0,T;H^{k+2}(\Omega))}^{2}+C\tau^{4}\norm{p_{ttt}(t)}^{2}_{L^{2}(0;T;L^{2}(\Omega))}\nonumber\\
			&&+C\tau^{4}\norm{p_{tt}}_{L^{1}(0,T;H^{1}(\Omega))}^{2}+C\tau^{4}\norm{p_{t}}_{L^{\infty}(0,T;H^{1}(\Omega))}^{2}.~~~~~~~
		\end{eqnarray}
	\end{thm}
	\begin{proof}
		Substituting $\hat{q}_{h}=\hat{\mathcal{E}}_{h}^{n+\frac{1}{2}}$ in the consistency error equation \eqref{erreqn}, we get
		\begin{eqnarray*}
			(\partial_{\tau} \mathcal{E}_{\mathcal{K}}^{n}, \mathcal{E}_{\mathcal{K}}^{n+\frac{1}{2}})+b(\hat{\mathcal{E}}_{h}^{n+\frac{1}{2}},\hat{\mathcal{E}}_{h}^{n+\frac{1}{2}})&=&\sum_{i=1}^{2}\mathcal{R}_{i}(p^{n+\frac{1}{2}},\hat{\mathcal{E}}_{h}^{n+\frac{1}{2}})-(p_{t}^{n+\frac{1}{2}}-\partial_{\tau}p^{n},\mathcal{E}_{\mathcal{K}}^{n+\frac{1}{2}})\nonumber\\
			&&+\mathcal{I}^{n+\frac{1}{2}}(\hat{\mathcal{E}}_{h}^{n+\frac{1}{2}})-\int_{0}^{t_{n+\frac{1}{2}}}b(\hat{I}_{h}^{k}(p(s)),\hat{\mathcal{E}}_{h}^{n+\frac{1}{2}})ds\nonumber\\
			&&+\int_{0}^{t_{n+\frac{1}{2}}}s_{h}(\hat{I}^{k}_{h}(p(s)),\hat{\mathcal{E}}_{h}^{n+\frac{1}{2}})ds\nonumber\\
			&&+s_{h}(\hat{I}^{k}_{h}(p^{n+\frac{1}{2}}),\hat{\mathcal{E}}_{h}^{n+\frac{1}{2}}).
		\end{eqnarray*}
		Rearranging the above equation, we have
		\begin{eqnarray}\label{e11}
			&&\norm{\mathcal{E}_{\mathcal{K}}^{n+1}}^{2}-\norm{\mathcal{E}_{\mathcal{K}}^{n}}^{2}+2\tau\norm{\hat{\mathcal{E}}_{h}^{n+\frac{1}{2}}}_{1,h}^{2}\nonumber\\
			&&=2\tau\sum_{i=1}^{2}\mathcal{R}_{i}(p^{n+\frac{1}{2}},\hat{\mathcal{E}}_{h}^{n+\frac{1}{2}})-2\tau(p_{t}^{n+\frac{1}{2}}-\partial_{\tau}p^{n},\mathcal{E}_{\mathcal{K}}^{n+\frac{1}{2}})\nonumber\\
			&&~~+2\tau\left(\mathcal{I}^{n+\frac{1}{2}}(\hat{\mathcal{E}}_{h}^{n+\frac{1}{2}})-\int_{0}^{t_{n+\frac{1}{2}}}b(\hat{I}_{h}^{k}(p(s)),\hat{\mathcal{E}}_{h}^{n+\frac{1}{2}})ds\right)\nonumber\\
			&&~~+2\tau s_{h}(\hat{I}^{k}_{h}(p^{n+\frac{1}{2}}),\hat{\mathcal{E}}_{h}^{n+\frac{1}{2}})+2\tau\int_{0}^{t_{n+\frac{1}{2}}}s_{h}(\hat{I}^{k}_{h}(p(s)),\hat{\mathcal{E}}_{h}^{n+\frac{1}{2}})ds.
		\end{eqnarray}
		Note that the third term in the right hand side of the above equation can be rewritten as		
		\begin{eqnarray}\label{e12}
			&&2\tau\left(\mathcal{I}^{n+\frac{1}{2}}(\hat{\mathcal{E}}_{h}^{n+\frac{1}{2}})-\int_{0}^{t_{n+\frac{1}{2}}}b(\hat{I}_{h}^{k}(p(s)),\hat{\mathcal{E}}_{h}^{n+\frac{1}{2}}ds\right)\nonumber\\
			&&=2\tau^{2}\sum_{j=0}^{n-1}b(\hat{P}_{h}^{j+\frac{1}{2}},\hat{\mathcal{E}}_{h}^{n+\frac{1}{2}})+\tau^{2}b(\hat{P}_{h}^{n+\frac{1}{2}},\hat{\mathcal{E}}_{h}^{n+\frac{1}{2}})\nonumber\\
			&&~~-2\tau\int_{0}^{t_{n+\frac{1}{2}}}b(\hat{I}_{h}^{k}(p(s)),\hat{\mathcal{E}}_{h}^{n+\frac{1}{2}})ds\nonumber\\
			&&=-2\tau^{2}\sum_{j=0}^{n-1}b(\hat{\mathcal{E}}_{h}^{j+\frac{1}{2}},\hat{\mathcal{E}}_{h}^{n+\frac{1}{2}})-\tau^{2}b(\hat{\mathcal{E}}_{h}^{n+\frac{1}{2}},\hat{\mathcal{E}}_{h}^{n+\frac{1}{2}})\nonumber\\
			&&~~+2\tau^{2}\sum_{j=0}^{n-1}b(\hat{I}_{h}^{k}(p(t_{j+\frac{1}{2}})),\hat{\mathcal{E}}_{h}^{n+\frac{1}{2}})+\tau^{2}b(\hat{I}_{h}^{k}(p(t_{n+\frac{1}{2}})),\hat{\mathcal{E}}_{h}^{n+\frac{1}{2}})\nonumber\\
			&&~~-2\tau\sum_{j=0}^{n-1}\int_{t_{j}}^{t_{j+1}}b(\hat{I}_{h}^{k}(p(s)),\hat{\mathcal{E}}_{h}^{n+\frac{1}{2}})ds-2\tau\int_{t_{n}}^{t_{n}+\frac{\tau}{2}}b(\hat{I}_{h}^{k}(p(s)),\hat{\mathcal{E}}_{h}^{n+\frac{1}{2}})ds\nonumber\\
			&&=-2\tau^{2}\sum_{j=0}^{n-1}b(\hat{\mathcal{E}}_{h}^{j+\frac{1}{2}},\hat{\mathcal{E}}_{h}^{n+\frac{1}{2}})-\tau^{2}b(\hat{\mathcal{E}}_{h}^{n+\frac{1}{2}},\hat{\mathcal{E}}_{h}^{n+\frac{1}{2}})\nonumber\\
			&&~~+2\tau\sum_{j=0}^{n-1}\int_{t_{j}}^{t_{j+1}}b(\hat{I}_{h}^{k}(p(t_{j+\frac{1}{2}}))-\hat{I}_{h}^{k}(p(s)),\hat{\mathcal{E}}_{h}^{n+\frac{1}{2}})ds\nonumber\\
			&&~~+2\tau\int_{t_{n}}^{t_{n}+\frac{\tau}{2}}b(\hat{I}_{h}^{k}(p(t_{n+\frac{1}{2}}))-\hat{I}_{h}^{k}(p(s)),\hat{\mathcal{E}}_{h}^{n+\frac{1}{2}})ds.
		\end{eqnarray}
		Using \eqref{e12} in \eqref{e11}, we observe
		\begin{eqnarray}\label{e13}
			&&\norm{\mathcal{E}_{\mathcal{K}}^{n+1}}^{2}-\norm{\mathcal{E}_{\mathcal{K}}^{n}}^{2}+(2\tau+\tau^{2})\norm{\hat{\mathcal{E}}_{h}^{n+\frac{1}{2}}}_{1,h}^{2}\nonumber\\
			&&=2\tau\sum_{i=1}^{2}\mathcal{R}_{i}(p^{n+\frac{1}{2}},\hat{\mathcal{E}}_{h}^{n+\frac{1}{2}})-2\tau(p_{t}^{n+\frac{1}{2}}-\partial_{\tau}p^{n}, \mathcal{E}_{\mathcal{K}}^{n+\frac{1}{2}})-2\tau^{2}\sum_{j=0}^{n-1}b(\hat{\mathcal{E}}_{h}^{j+\frac{1}{2}},\hat{\mathcal{E}}_{h}^{n+\frac{1}{2}})\nonumber\\
			&&~~+2\tau\sum_{j=0}^{n-1}\int_{t_{j}}^{t_{j+1}}b(\hat{I}_{h}^{k}(p(t_{j+\frac{1}{2}}))-\hat{I}_{h}^{k}(p(s)),\hat{\mathcal{E}}_{h}^{n+\frac{1}{2}})ds\nonumber\\
			&&~~+2\tau\int_{t_{n}}^{t_{n+\frac{\tau}{2}}}b(\hat{I}_{h}^{k}(p(t_{n+\frac{1}{2}}))-\hat{I}_{h}^{k}(p(s)),\hat{\mathcal{E}}_{h}^{n+\frac{1}{2}})ds\nonumber\\
			&&~~+2\tau s_{h}(\hat{I}^{k}_{h}(p^{n+\frac{1}{2}}),\hat{\mathcal{E}}_{h}^{n+\frac{1}{2}})+2\tau\int_{0}^{t_{n+\frac{1}{2}}}s_{h}(\hat{I}^{k}_{h}(p(s)),\hat{\mathcal{E}}_{h}^{n+\frac{1}{2}})ds.
		\end{eqnarray}
		We now proceed to find an appropriate estimate for each of the terms in right hand side of \eqref{e13}.
		
		From Lemma \ref{bounds} and Young's inequality, we determine
		\begin{eqnarray}\label{e14}
			\left|2\tau\sum_{i=1}^{2}\mathcal{R}_{i}(p^{n+\frac{1}{2}},\hat{\mathcal{E}}_{h}^{n+\frac{1}{2}})\right|&\le&C\tau h^{k+1}\norm{p^{n+\frac{1}{2}}}_{k+2}\norm{\hat{\mathcal{E}}_{h}^{n+\frac{1}{2}}}_{1,h}\nonumber\\
			&\le&C\tau h^{2k+2}\norm{p^{n+\frac{1}{2}}}_{H^{k+2}(\Omega)}^{2}+\frac{\tau}{6}\norm{\hat{\mathcal{E}}_{h}^{n+\frac{1}{2}}}_{1,h}^{2}.~~~
		\end{eqnarray}
		We have the following truncation error bound (see inequality (3.4) of \cite{zhu2022supercloseness}).
		\begin{equation}\label{trunc}
			\norm{p_{t}^{n+\frac{1}{2}}-\partial_{\tau}p^{n}}\le C\tau^{\frac{3}{2}}\left(\int_{t_{n}}^{t_{n+1}}\norm{p_{ttt}(t)}^{2}dt\right)^{\frac{1}{2}}.
		\end{equation}
		From the above truncation error bound \eqref{trunc}, Cauchy-Schwarz inequality, discrete Poincar\'e estimate \eqref{Poincare} and triangle inequality, we achieve
		\begin{eqnarray}\label{e15}
			\left|2\tau(p_{t}^{n+\frac{1}{2}}-\partial_{\tau}p^{n},\mathcal{E}_{\mathcal{K}}^{n+\frac{1}{2}})\right|&\le&2\tau\norm{p_{t}^{n+\frac{1}{2}}-\partial_{\tau}p^{n}}\norm{\mathcal{E}_{\mathcal{K}}^{n+\frac{1}{2}}}\nonumber\\
			&\le&C\tau^{\frac{5}{2}}\left(\int_{t_{n}}^{t_{n+1}}\norm{p_{ttt}(t)}^{2}dt\right)^{\frac{1}{2}}\left(\norm{\mathcal{E}^{n+1}_{\mathcal{K}}}+\norm{\mathcal{E}_{\mathcal{K}}^{n}}\right)\nonumber\\
			&\le&C\tau^{4}\int_{t_{n}}^{t_{n+1}}\norm{p_{ttt}(t)}^{2}dt+\tau\left(\norm{\mathcal{E}_{\mathcal{K}}^{n+1}}^{2}+\norm{\mathcal{E}_{\mathcal{K}}^{n}}^{2}\right).\nonumber\\
		\end{eqnarray}		
		It follows from boundedness property of the bilinear form (see \eqref{boundedness}) and Cauchy-Schwarz inequality that
		\begin{eqnarray}\label{e16}
			\left|2\tau^{2}\sum_{j=0}^{n-1}b(\hat{\mathcal{E}}_{h}^{j+\frac{1}{2}},\hat{\mathcal{E}}_{h}^{n+\frac{1}{2}})\right|&\le&2\tau^{2}\norm{\hat{\mathcal{E}}_{h}^{n+\frac{1}{2}}}_{1,h}\sum_{j=0}^{n-1}\norm{\hat{\mathcal{E}}_{h}^{j+\frac{1}{2}}}_{1,h}\nonumber\\
			&\le&\frac{\tau}{6}\norm{\hat{\mathcal{E}}_{h}^{n+\frac{1}{2}}}_{1,h}^{2}+Cn\tau^{3}\sum_{j=0}^{n-1}\norm{\hat{\mathcal{E}}_{h}^{j+\frac{1}{2}}}_{1,h}^{2}\nonumber\\
			&\le&\frac{\tau}{6}\norm{\hat{\mathcal{E}}_{h}^{n+\frac{1}{2}}}_{1,h}^{2}+C\tau^{2}\sum_{j=0}^{n-1}\norm{\hat{\mathcal{E}}_{h}^{j+\frac{1}{2}}}_{1,h}^{2}.
		\end{eqnarray}
		Utilizing linearity property of $\hat{I}_{h}^{k}$, next using boundedness properties of the bilinear form (see \eqref{boundedness}) and $\hat{I}_{h}^{k}$ (see Lemma \ref{boundedprojector}), the truncation error bound (see Lemma \ref{Taylor1}) and finally by Young's inequality, we derive
		\begin{eqnarray}\label{e17}
			&&\left|2\tau\sum_{j=0}^{n-1}\int_{t_{j}}^{t_{j+1}}b(\hat{I}_{h}^{k}(p(t_{j+\frac{1}{2}}))-\hat{I}_{h}^{k}(p(s)),\hat{\mathcal{E}}_{h}^{n+\frac{1}{2}})ds\right|\nonumber\\
			&&=2\tau\left|b\left(\hat{I}_{h}^{k}\left(\sum_{j=0}^{n-1}\int_{t_{j}}^{t_{j+1}}(p(t_{j+\frac{1}{2}})-p(s))ds\right),\hat{\mathcal{E}}_{h}^{n+\frac{1}{2}}\right)\right|\nonumber\\
			&&\le2\tau\left\|\hat{I}_{h}^{k}\left(\sum_{j=0}^{n-1}\int_{t_{j}}^{t_{j+1}}(p(t_{j+\frac{1}{2}})-p(s))ds\right)\right\|_{1,h}\norm{\hat{\mathcal{E}}_{h}^{n+\frac{1}{2}}}_{1,h}\nonumber\\
			&&\le2\tau\left\|\sum_{j=0}^{n-1}\int_{t_{j}}^{t_{j+1}}(p(t_{j+\frac{1}{2}})-p(s))ds\right\|_{H^{1}(\Omega)}\norm{\hat{\mathcal{E}}_{h}^{n+\frac{1}{2}}}_{1,h}\nonumber\\
			&&\le2\tau\left(\sum_{j=0}^{n-1}\left\|\int_{t_{j}}^{t_{j+1}}(p(t_{j+\frac{1}{2}})-p(s))ds\right\|_{H^{1}(\Omega)}\right)\norm{\hat{\mathcal{E}}_{h}^{n+\frac{1}{2}}}_{1,h}\nonumber\\
			&&\le\frac{\tau^{3}}{2}\left(\sum_{j=0}^{n-1}\int_{t_{j}}^{t_{j+1}}\norm{p_{\xi\xi}(\xi)}_{H^{1}(\Omega)}d\xi\right)\norm{\hat{\mathcal{E}}_{h}^{n+\frac{1}{2}}}_{1,h}\nonumber\\
			&&\le C\tau^{3}\norm{p_{tt}}_{L^{1}(0,T;H^{1}(\Omega))}\norm{\hat{\mathcal{E}}_{h}^{n+\frac{1}{2}}}_{1,h}\nonumber\\
			&&\le C\tau^{5}\norm{p_{tt}}_{L^{1}(0,T;H^{1}(\Omega))}^{2}+\frac{\tau}{6}\norm{\hat{\mathcal{E}}_{h}^{n+\frac{1}{2}}}_{1,h}^{2}.
		\end{eqnarray}
		Arguing as above by applying the truncation error bound (see Lemma \ref{Taylor2}), we determine 
		\begin{eqnarray}\label{e18}
			&&\left|2\tau\int_{t_{n}}^{t_{n}+\frac{\tau}{2}}b(\hat{I}_{h}^{k}(p(t_{n+\frac{1}{2}}))-\hat{I}_{h}^{k}(p(s)),\hat{\mathcal{E}}_{h}^{n+\frac{1}{2}})ds\right|\nonumber\\
			&&\le2\tau\left\|\int_{t_{n}}^{t_{n}+\frac{\tau}{2}}(p(t_{n+\frac{1}{2}})-p(s))ds\right\|_{H^{1}(\Omega)}\norm{\hat{\mathcal{E}}_{h}^{n+\frac{1}{2}}}_{1,h}\nonumber\\			
			&&\le C\tau^{\frac{5}{2}}\left(\int_{t_{n}}^{t_{n}+\frac{\tau}{2}}\norm{p_{\xi}(\xi)}_{H^{1}(\Omega)}^{2} d\xi\right)^{\frac{1}{2}} \norm{\hat{\mathcal{E}}_{h}^{n+\frac{1}{2}}}_{1,h}\nonumber\\
			&&\le C\tau^{3}\norm{p_{t}}_{L^{\infty}(0,T;H^{1}(\Omega))} \norm{\hat{\mathcal{E}}_{h}^{n+\frac{1}{2}}}_{1,h}\nonumber\\
			&&\le C\tau^{5}\norm{p_{t}}_{L^{\infty}(0,T;H^{1}(\Omega))}^{2}+\frac{\tau}{6} \norm{\hat{\mathcal{E}}_{h}^{n+\frac{1}{2}}}_{1,h}^{2}.
		\end{eqnarray}
		It follows from Lemma \ref{bounds} and Young's inequality that
		\begin{eqnarray}\label{e19}
			\left|2\tau s_{h}(\hat{I}^{k}_{h}(p^{n+\frac{1}{2}}),\hat{\mathcal{E}}_{h}^{n+\frac{1}{2}})\right|&\le& C\tau h^{k+1}\norm{p^{n+\frac{1}{2}}}_{H^{k+2}(\Omega)}\norm{\hat{\mathcal{E}}_{h}^{n+\frac{1}{2}}}_{1,h}\nonumber\\
			&\le&C\tau h^{2k+2}\norm{p^{n+\frac{1}{2}}}_{H^{k+2}(\Omega)}^{2}+\frac{\tau}{6}\norm{\hat{\mathcal{E}}_{h}^{n+\frac{1}{2}}}_{1,h}^{2}.~~
		\end{eqnarray}
		Proceeding in the same way as above, we deduce
		\begin{eqnarray}\label{e20}
			\left|2\tau\int_{0}^{t_{n+\frac{1}{2}}}s_{h}(\hat{I}^{k}_{h}(p(s)),\hat{\mathcal{E}}_{h}^{n+\frac{1}{2}})ds\right|
			&=& 2\tau\left| s_{h}\left(\hat{I}^{k}_{h}\left(\int_{0}^{t_{n+\frac{1}{2}}}p(s)ds\right),\hat{\mathcal{E}}_{h}^{n+\frac{1}{2}}\right)\right|\nonumber\\
			&\le& C\tau h^{k+1}\norm{\int_{0}^{t_{n+\frac{1}{2}}}p(s)ds}_{1,h}\norm{\hat{\mathcal{E}}_{h}^{n+\frac{1}{2}}}_{1,h}\nonumber\\
			&\le& C\tau h^{k+1}\norm{p}_{L^{\infty}(0,T;H^{k+2}(\Omega))}\norm{\hat{\mathcal{E}}_{h}^{n+\frac{1}{2}}}_{1,h}\nonumber\\
			&\le& C\tau h^{2k+2}\norm{p}_{L^{\infty}(0,T;H^{k+2}(\Omega))}^{2}+\frac{\tau}{6}\norm{\hat{\mathcal{E}}_{h}^{n+\frac{1}{2}}}_{1,h}^{2}.\nonumber\\
		\end{eqnarray}
		Combining the estimates \eqref{e14}-\eqref{e20} and using them in \eqref{e13}, we have
		\begin{eqnarray}\label{e21}
			&&\norm{\mathcal{E}_{\mathcal{K}}^{n+1}}^{2}-\norm{\mathcal{E}_{\mathcal{K}}^{n}}^{2}+2\tau\norm{\hat{\mathcal{E}}_{h}^{n+\frac{1}{2}}}_{1,h}^{2}\nonumber\\
			&&\le C\tau h^{2k+2}\norm{p^{n+\frac{1}{2}}}_{H^{k+2}(\Omega)}^{2}+C\tau^{4}\int_{t_{n}}^{t_{n+1}}\norm{p_{ttt}(t)}^{2}dt+\tau\left(\norm{\mathcal{E}_{\mathcal{K}}^{n+1}}^{2}+\norm{\mathcal{E}_{\mathcal{K}}^{n}}^{2}\right)\nonumber\\
			&&~~+C\tau^{2}\sum_{j=0}^{n-1}\norm{\hat{\mathcal{E}}_{h}^{j+\frac{1}{2}}}_{1,h}^{2}+C\tau^{5}\norm{p_{tt}}_{L^{1}(0,T;H^{1}(\Omega))}^{2}+C\tau^{5}\norm{p_{t}}_{L^{\infty}(0,T;H^{1}(\Omega))}^{2}\nonumber\\
			&&~~+C\tau h^{2k+2}\norm{p^{n+\frac{1}{2}}}_{H^{k+2}(\Omega)}^{2}+C\tau h^{2k+2}\norm{p}_{L^{\infty}(0,T;H^{k+2}(\Omega))}^{2}+\tau\norm{\hat{\mathcal{E}}_{h}^{n+\frac{1}{2}}}_{1,h}^{2}.~~~~~~~~
		\end{eqnarray}
		Summing the above equation \eqref{e21} from $n=0$ to $l$,  $l\in\{0,1,2,\ldots,M-1\}$, and using the fact $\mathcal{E}_{\mathcal{K}}^{0}=0$, we derive
		\begin{eqnarray}\label{e22}
			&&\norm{\mathcal{E}_{\mathcal{K}}^{l+1}}^{2}+\tau\sum_{n=0}^{l}\norm{\hat{\mathcal{E}}_{h}^{n+\frac{1}{2}}}_{1,h}^{2}\nonumber\\
			&&\le C\tau h^{2k+2}\sum_{n=0}^{l}\norm{p^{n+\frac{1}{2}}}_{H^{k+2}(\Omega)}^{2}+C\tau^{4}\int_{0}^{T}\norm{p_{ttt}(t)}^{2}dt+C\tau^{4}\norm{p_{tt}}_{L^{1}(0,T;H^{1}(\Omega))}^{2}\nonumber\\
			&&~~+C\tau^{4}\norm{p_{t}}_{L^{\infty}(0,T;H^{1}(\Omega))}^{2}+C h^{2k+2}\norm{p}_{L^{\infty}(0,T;H^{k+2}(\Omega))}^{2}\nonumber\\
			&&~~+C\tau\sum_{n=0}^{l+1}\left(\norm{\hat{\mathcal{E}}_{h}^{n}}^{2}+\tau\sum_{j=0}^{n-1}\norm{\hat{\mathcal{E}}_{h}^{j+\frac{1}{2}}}_{1,h}^{2}\right)\nonumber\\
			&&\le C h^{2k+2}\norm{p}_{L^{\infty}(0;T;H^{k+2}(\Omega))}^{2}+C\tau^{4}\norm{p_{ttt}}_{L^{2}(0;T;L^{2}(\Omega))}^{2}+C\tau^{4}\norm{p_{tt}}_{L^{1}(0,T;H^{1}(\Omega))}^{2}\nonumber\\
			&&~~+C\tau^{4}\norm{p_{t}}_{L^{\infty}(0,T;H^{1}(\Omega))}^{2}+C\tau\sum_{n=0}^{l+1}\left(\norm{\hat{\mathcal{E}}_{h}^{n}}^{2}+\tau\sum_{j=0}^{n-1}\norm{\hat{\mathcal{E}}_{h}^{j+\frac{1}{2}}}_{1,h}^{2}\right).
		\end{eqnarray}
		We obtain the desired error estimate \eqref{errorest1} from \eqref{e22} by simply applying the discrete Gronwall's lemma.
	\end{proof}
		\begin{thm}\label{maxH1}
				Under the regularity assumptions,
			\begin{eqnarray*}
				&&p\in  H^{1}(0,T;H^{k+2}(\Omega))\cap  W^{2,\infty}(0,T;H^{1}(\Omega)),\\
				&&~~~~~~~~~~~~p_{ttt}\in  L^{2}(0;T;L^{2}(\Omega)).
			\end{eqnarray*} 
			Then for the complete discrete consistency error, we have the following estimate.
			\begin{eqnarray}\label{maxH11}
			\max_{0\le l\le M-1}\norm{\hat{\mathcal{E}}_{h}^{l+1}}_{1,h}^{2}
			&\le& C_{1}(p)h^{2k+2}+C_{2}(p)\tau^{2}
		    \end{eqnarray}
		     Here,
		     \begin{eqnarray*}
		     	C_{1}(p)&=&C(\norm{p}_{L^{\infty}(0,T;H^{k+2}(\Omega))}^{2}+\norm{p}_{L^{2}(0,T;H^{k+2}(\Omega))}^{2}+\norm{p_{t}}_{L^{2}(0,T;H^{k+2}(\Omega))}^{2}),\nonumber\\
		     	C_{2}(p)&=&C\left(\norm{p_{t}}_{L^{\infty}(0,T;H^{1}(\Omega))}^{2}+\norm{p_{tt}}_{L^{2}(0,T;H^{1}(\Omega))}^{2}+\norm{p_{tt}}_{L^{\infty}(0,T;H^{1}(\Omega))}^{2}+\norm{p_{ttt}}_{L^{2}(0,T;L^{2}(\Omega))}^{2}\right).
		     \end{eqnarray*}
	\end{thm}
	\begin{proof}
		We begin the proof by making the substitution $\hat{q}_{h}=\partial_{\tau}\hat{\mathcal{E}}_{h}^{n}$ in the error equation \eqref{erreqn} to get
		\begin{eqnarray*}
			(\partial_{\tau} \mathcal{E}_{\mathcal{K}}^{n},\partial_{\tau} \mathcal{E}_{\mathcal{K}}^{n})+b(\hat{\mathcal{E}}_{h}^{n+\frac{1}{2}},\partial_{\tau}\hat{\mathcal{E}}_{h}^{n})&=&\sum_{i=1}^{2}\mathcal{R}_{i}(p^{n+\frac{1}{2}},\partial_{\tau}\hat{\mathcal{E}}_{h}^{n})+(p_{t}^{n+\frac{1}{2}}-\partial_{\tau}p^{n},\partial_{\tau} \mathcal{E}_{\mathcal{K}}^{n})\nonumber\\
			&&+\mathcal{I}^{n+\frac{1}{2}}(\partial_{\tau}\hat{\mathcal{E}}_{h}^{n})-\int_{0}^{t_{n+\frac{1}{2}}}b(\hat{I}_{h}^{k}(p(s)),\partial_{\tau}\hat{\mathcal{E}}_{h}^{n})ds\nonumber\\
			&&+\int_{0}^{t_{n+\frac{1}{2}}}s_{h}(\hat{I}^{k}_{h}(p(s)),\partial_{\tau}\hat{\mathcal{E}}_{h}^{n})ds\nonumber\\
			&&+s_{h}(\hat{I}^{k}_{h}(p^{n+\frac{1}{2}}),\partial_{\tau}\hat{\mathcal{E}}_{h}^{n}).
		\end{eqnarray*}
		Simplifying the above equation, we get
		\begin{eqnarray}\label{e2}
			&&2\tau\norm{\partial_{\tau} \mathcal{E}_{\mathcal{K}}^{n}}^{2}+\norm{\hat{\mathcal{E}}_{h}^{n+1}}_{1,h}^{2}-\norm{\hat{\mathcal{E}}_{h}^{n}}_{1,h}^{2}\nonumber\\
			&&=2\tau\sum_{i=1}^{2}\mathcal{R}_{i}(p^{n+\frac{1}{2}},\partial_{\tau}\hat{\mathcal{E}}_{h}^{n})+2\tau(p_{t}^{n+\frac{1}{2}}-\partial_{\tau}p^{n},\partial_{\tau} \mathcal{E}_{\mathcal{K}}^{n})\nonumber\\
			&&~~+2\tau\left(\mathcal{I}^{n+\frac{1}{2}}(\partial_{\tau}\hat{\mathcal{E}}_{h}^{n})-\int_{0}^{t_{n+\frac{1}{2}}}b(\hat{I}_{h}^{k}(p(s)),\partial_{\tau}\hat{\mathcal{E}}_{h}^{n})ds\right)\nonumber\\
			&&~~+2\tau s_{h}(\hat{I}^{k}_{h}(p^{n+\frac{1}{2}}),\partial_{\tau}\hat{\mathcal{E}}_{h}^{n})+2\tau\int_{0}^{t_{n+\frac{1}{2}}}s_{h}(\hat{I}^{k}_{h}(p(s)),\partial_{\tau}\hat{\mathcal{E}}_{h}^{n})ds.
		\end{eqnarray}
		Arguing as in \eqref{e12}, the third term in the right hand side of the above equation can be rewritten as		
			\begin{eqnarray}\label{e3}
			&&2\tau\left(\mathcal{I}^{n+\frac{1}{2}}(\partial_{\tau}\hat{\mathcal{E}}_{h}^{n})-\int_{0}^{t_{n+\frac{1}{2}}}b(\hat{I}_{h}^{k}(p(s)),\partial_{\tau}\hat{\mathcal{E}}_{h}^{n})ds\right)\nonumber\\
			&&=-2\tau^{2}\sum_{j=0}^{n-1}b(\hat{\mathcal{E}}_{h}^{j+\frac{1}{2}},\partial_{\tau}\hat{\mathcal{E}}_{h}^{n})-\tau^{2} b(\hat{\mathcal{E}}_{h}^{n+\frac{1}{2}},\partial_{\tau}\hat{\mathcal{E}}_{h}^{n})\nonumber\\
			&&~~+2\tau\sum_{j=0}^{n-1}\int_{t_{j}}^{t_{j+1}}b(\hat{I}_{h}^{k}(p(t_{j+\frac{1}{2}}))-\hat{I}_{h}^{k}(p(s)),\partial_{\tau}\hat{\mathcal{E}}_{h}^{n})ds\nonumber\\
			&&~~+2\tau\int_{t_{n}}^{t_{n+\frac{\tau}{2}}}b(\hat{I}_{h}^{k}(p(t_{n+\frac{1}{2}}))-\hat{I}_{h}^{k}(p(s)),\partial_{\tau}\hat{\mathcal{E}}_{h}^{n})ds.
		\end{eqnarray}
		Using \eqref{e2} in \eqref{e3}, we obtain
		\begin{eqnarray*}
			&&2\tau\norm{\partial_{\tau} \mathcal{E}_{\mathcal{K}}^{n}}^{2}+\norm{\hat{\mathcal{E}}_{h}^{n+1}}_{1,h}^{2}-\norm{\hat{\mathcal{E}}_{h}^{n}}_{1,h}^{2}\nonumber\\
			&&=2\tau\sum_{i=1}^{2}\mathcal{R}_{i}(p^{n+\frac{1}{2}},\partial_{\tau}\hat{\mathcal{E}}_{h}^{n})+2\tau(p_{t}^{n+\frac{1}{2}}-\partial_{\tau}p^{n},\partial_{\tau} \mathcal{E}_{\mathcal{K}}^{n})-2\tau^{2}\sum_{j=0}^{n-1}b(\hat{\mathcal{E}}_{h}^{j+\frac{1}{2}},\partial_{\tau}\hat{\mathcal{E}}_{h}^{n})\nonumber\\
			&&~~-\tau^{2} b(\hat{\mathcal{E}}_{h}^{n+\frac{1}{2}},\partial_{\tau}\hat{\mathcal{E}}_{h}^{n})+2\tau\sum_{j=0}^{n-1}\int_{t_{j}}^{t_{j+1}}b(\hat{I}_{h}^{k}(p(t_{j+\frac{1}{2}}))-\hat{I}_{h}^{k}(p(s)),\partial_{\tau}\hat{\mathcal{E}}_{h}^{n})ds\nonumber\\
			&&~~+2\tau\int_{t_{n}}^{t_{n+\frac{\tau}{2}}}b(\hat{I}_{h}^{k}(p(t_{n+\frac{1}{2}}))-\hat{I}_{h}^{k}(p(s)),\partial_{\tau}\hat{\mathcal{E}}_{h}^{n})ds\nonumber\\
			&&~~+2\tau s_{h}(\hat{I}^{k}_{h}(p^{n+\frac{1}{2}}),\partial_{\tau}\hat{\mathcal{E}}_{h}^{n})+2\tau\int_{0}^{t_{n+\frac{1}{2}}}s_{h}(\hat{I}^{k}_{h}(p(s)),\partial_{\tau}\hat{\mathcal{E}}_{h}^{n})ds.
		\end{eqnarray*}
		A simple application of Young's inequality to the above equation yields,
		\begin{eqnarray*}
			&&2\tau\norm{\partial_{\tau} \mathcal{E}_{\mathcal{K}}^{n}}^{2}+\norm{\hat{\mathcal{E}}_{h}^{n+1}}_{1,h}^{2}-\norm{\hat{\mathcal{E}}_{h}^{n}}_{1,h}^{2}\nonumber\\
			&&\le2\tau\sum_{i=1}^{2}\mathcal{R}_{i}(p^{n+\frac{1}{2}},\partial_{\tau}\hat{\mathcal{E}}_{h}^{n})+\tau\norm{p_{t}^{n+\frac{1}{2}}-\partial_{\tau}p^{n}}^{2}+\tau\norm{\partial_{\tau}^{2} \mathcal{E}_{\mathcal{K}}^{n}}-2\tau^{2}\sum_{j=0}^{n-1}b(\hat{\mathcal{E}}_{h}^{j+\frac{1}{2}},\partial_{\tau}\hat{\mathcal{E}}_{h}^{n})\nonumber\\
			&&~~-\tau^{2} b(\hat{\mathcal{E}}_{h}^{n+\frac{1}{2}},\partial_{\tau}\hat{\mathcal{E}}_{h}^{n})+2\tau\sum_{j=0}^{n-1}\int_{t_{j}}^{t_{j+1}}b(\hat{I}_{h}^{k}(p(t_{j+\frac{1}{2}}))-\hat{I}_{h}^{k}(p(s)),\partial_{\tau}\hat{\mathcal{E}}_{h}^{n})ds\nonumber\\
			&&~~+2\tau\int_{t_{n}}^{t_{n+\frac{\tau}{2}}}b(\hat{I}_{h}^{k}(p(t_{n+\frac{1}{2}}))-\hat{I}_{h}^{k}(p(s)),\partial_{\tau}\hat{\mathcal{E}}_{h}^{n})ds\nonumber\\
			&&~~+2\tau s_{h}(\hat{I}^{k}_{h}(p^{n+\frac{1}{2}}),\partial_{\tau}\hat{\mathcal{E}}_{h}^{n})+2\tau\int_{0}^{t_{n+\frac{1}{2}}}s_{h}(\hat{I}^{k}_{h}(p(s)),\partial_{\tau}\hat{\mathcal{E}}_{h}^{n})ds.
		\end{eqnarray*}
		Summing over $n=0$ to $l$, $(l=\{0,1,2,\ldots,N-1\})$, on both sides of above inequality and using the fact $\hat{\mathcal{E}}_{h}^{0}=0$, we get
		\begin{eqnarray}\label{e4}
			&&\tau\sum_{n=0}^{l}\norm{\partial_{\tau} \mathcal{E}_{\mathcal{K}}^{n}}^{2}+\norm{\hat{\mathcal{E}}_{h}^{l+1}}_{1,h}^{2}\nonumber\\
			&&\le2\tau\sum_{i=1}^{2}\sum_{n=0}^{l}\mathcal{R}_{i}(p^{n+\frac{1}{2}},\partial_{\tau}\hat{\mathcal{E}}_{h}^{n})+\tau\sum_{n=0}^{l}\norm{p_{t}^{n+\frac{1}{2}}-\partial_{\tau}p^{n}}^{2}-2\tau^{2}\sum_{n=0}^{l}\sum_{j=0}^{n-1}b(\hat{\mathcal{E}}_{h}^{j+\frac{1}{2}},\partial_{\tau}\hat{\mathcal{E}}_{h}^{n})\nonumber\\
			&&~~-\tau^{2}\sum_{n=0}^{l} b(\hat{\mathcal{E}}_{h}^{n+\frac{1}{2}},\partial_{\tau}\hat{\mathcal{E}}_{h}^{n})+2\tau\sum_{n=0}^{l}\sum_{j=0}^{n-1}\int_{t_{j}}^{t_{j+1}}b(\hat{I}_{h}^{k}(p(t_{j+\frac{1}{2}}))-\hat{I}_{h}^{k}(p(s)),\partial_{\tau}\hat{\mathcal{E}}_{h}^{n})ds\nonumber\\
			&&~~+2\tau\sum_{n=0}^{l}\int_{t_{n}}^{t_{n+\frac{\tau}{2}}}b(\hat{I}_{h}^{k}(p(t_{n+\frac{1}{2}}))-\hat{I}_{h}^{k}(p(s)),\partial_{\tau}\hat{\mathcal{E}}_{h}^{n})ds\nonumber\\
			&&~~+2\tau\sum_{n=0}^{l} s_{h}(\hat{I}^{k}_{h}(p^{n+\frac{1}{2}}),\partial_{\tau}\hat{\mathcal{E}}_{h}^{n})+2\tau\sum_{n=0}^{l}\int_{0}^{t_{n+\frac{1}{2}}}s_{h}(\hat{I}^{k}_{h}(p(s)),\partial_{\tau}\hat{\mathcal{E}}_{h}^{n})ds.
		\end{eqnarray}
		The next task boils down in finding appropriate bounds of each terms in the right hand side of \eqref{e4}.
		
		For $i=\{1,2\}$, using Lemma \ref{bounds} and Young's inequality, we achieve 
	    \begin{eqnarray}\label{e5}
	    	\left|2\tau\sum_{n=0}^{l}\mathcal{R}_{i}(p^{n+\frac{1}{2}},\partial_{\tau}\hat{\mathcal{E}}_{h}^{n})\right|&=&\left|\tau\sum_{n=0}^{l}\mathcal{R}_{i}(-\partial_{\tau}p^{n},\hat{\mathcal{E}}_{h}^{n+\frac{1}{2}})+2\mathcal{R}_{i}(p^{l+1},\hat{\mathcal{E}}_{h}^{l+1})\right|\nonumber\\
	    	&\le& C\tau h^{k+1}\sum_{n=0}^{l}\norm{\partial_{\tau}p^{n}}_{H^{k+2}(\Omega)}\norm{\hat{\mathcal{E}}_{h}^{n+\frac{1}{2}}}_{1,h}\nonumber\\
	    	&&+C h^{k+1}\norm{p^{l+1}}_{H^{k+2}(\Omega)}\norm{\hat{\mathcal{E}}_{h}^{l+\frac{1}{2}}}_{1,h}\nonumber\\
	    	&\le& C\tau h^{2k+2}\sum_{n=0}^{l}\norm{\partial_{\tau}p^{n}}_{H^{k+2}(\Omega)}^{2}+C\tau\sum_{n=0}^{l}\norm{\hat{\mathcal{E}}_{h}^{n+\frac{1}{2}}}_{1,h}^{2}\nonumber\\
	    	&&+C h^{2k+2}\norm{p^{l+1}}_{H^{k+2}(\Omega)}^{2}+\frac{1}{10}\norm{\hat{\mathcal{E}}_{h}^{l+\frac{1}{2}}}_{1,h}^{2}\nonumber\\
	        &\le& C h^{2k+2}\norm{p_{t}}_{L^{2}(0,T;H^{k+2}(\Omega))}^{2}+C\tau\sum_{n=0}^{l}\norm{\hat{\mathcal{E}}_{h}^{n+\frac{1}{2}}}_{1,h}^{2}\nonumber\\
	    	&&+C h^{2k+2}\norm{p}_{L^{\infty}(0,T;H^{k+2}(\Omega))}^{2}+\frac{1}{10}\norm{\hat{\mathcal{E}}_{h}^{l+\frac{1}{2}}}_{1,h}^{2}.~~~~~~~~
	    \end{eqnarray}
	    In the last inequality, we have used estimate (3.12) of \cite{zhu2022supercloseness}.
	    Similarly, we can show
	    \begin{eqnarray}\label{e6}
	    	\left|2\tau\sum_{n=0}^{l} s_{h}(\hat{I}^{k}_{h}(p^{n+\frac{1}{2}}),\partial_{\tau}\hat{\mathcal{E}}_{h}^{n})\right|&\le& C h^{2k+2}\norm{p_{t}}_{L^{2}(0,T;H^{k+2}(\Omega))}^{2}+C\tau\sum_{n=0}^{l}\norm{\hat{\mathcal{E}}_{h}^{n+\frac{1}{2}}}_{1,h}^{2}\nonumber\\
	    	&&+C h^{2k+2}\norm{p}_{L^{\infty}(0,T;H^{k+2}(\Omega))}^{2}+\frac{1}{10}\norm{\hat{\mathcal{E}}_{h}^{l+\frac{1}{2}}}_{1,h}^{2}.\nonumber\\
	    \end{eqnarray}
	    From \eqref{trunc}, it is easy to see
	    \begin{eqnarray}\label{e7}
	    	\tau\sum_{n=0}^{l}\norm{p_{t}^{n+\frac{1}{2}}-\partial_{\tau}p^{n}}^{2}\le C\tau^{4}\norm{p_{ttt}}_{L^{2}(0,T;L^{2}(\Omega))}^{2}.
	    \end{eqnarray}
	    Set	$L_{n}=\sum_{j=0}^{n-1}\hat{\mathcal{E}}_{h}^{j+\frac{1}{2}}$. This implies
	    \begin{eqnarray}\label{e8*}
	    	L_{n}-L_{n-1}=\sum_{j=0}^{n-1}\hat{\mathcal{E}}_{h}^{j+\frac{1}{2}}-\sum_{j=0}^{n-2}\hat{\mathcal{E}}_{h}^{j+\frac{1}{2}}=\hat{\mathcal{E}}_{h}^{n-\frac{1}{2}}.
	    \end{eqnarray}
	    We use these identities in deriving the subsequent inequality. Some straightforward calculations, the relation \eqref{e8*}, boundedness property of bilinear form (see \eqref{boundedness}), Cauchy-Schwarz inequality and Young's inequality leads to
	    \begin{eqnarray}\label{e8}
	    	&&\left|-2\tau^{2}\sum_{n=0}^{l}\sum_{j=0}^{n-1}b(\hat{\mathcal{E}}_{h}^{j+\frac{1}{2}},\partial_{\tau}\hat{\mathcal{E}}_{h}^{n})\right|\nonumber\\
	    	&&=\left|2\tau\sum_{n=0}^{l}b(L_{n},\hat{\mathcal{E}}_{h}^{n+1}-\hat{\mathcal{E}}_{h}^{n})\right|\nonumber\\
	    	&&=\left|2\tau b(L_{l},\hat{\mathcal{E}}_{h}^{l+1})+2\tau\sum_{n=1}^{l}b(L_{n-1}-L_{n},\hat{\mathcal{E}}_{h}^{n})\right|\nonumber\\
	    	&&=\left|2\tau b\left(\sum_{j=0}^{l-1}\hat{\mathcal{E}}_{h}^{j+\frac{1}{2}},\hat{\mathcal{E}}_{h}^{l+1}\right)-2\tau\sum_{n=1}^{l}b(\hat{\mathcal{E}}_{h}^{n-\frac{1}{2}},\hat{\mathcal{E}}_{h}^{n})\right|\nonumber\\
	    	&&\le2\tau\sum_{j=0}^{l-1} \norm{\hat{\mathcal{E}}_{h}^{j+\frac{1}{2}}}_{1,h}\norm{\hat{\mathcal{E}}_{h}^{l+1}}_{1,h}+2\tau\sum_{n=1}^{l}\norm{\hat{\mathcal{E}}_{h}^{n-\frac{1}{2}}}_{1,h}\norm{\hat{\mathcal{E}}_{h}^{n}}_{1,h}\nonumber\\
	    		&&\le C\tau\sum_{j=0}^{l-1} \norm{\hat{\mathcal{E}}_{h}^{j+\frac{1}{2}}}_{1,h}^{2}+\frac{1}{10}\norm{\hat{\mathcal{E}}_{h}^{l+1}}_{1,h}+C\tau\sum_{n=1}^{l}\norm{\hat{\mathcal{E}}_{h}^{n-\frac{1}{2}}}_{1,h}^{2}+C\tau\sum_{n=1}^{l}\norm{\hat{\mathcal{E}}_{h}^{n}}_{1,h}^{2}\nonumber\\
	    		&&\le C\tau\sum_{n=0}^{l} \norm{\hat{\mathcal{E}}_{h}^{n+\frac{1}{2}}}_{1,h}^{2}+\frac{1}{10}\norm{\hat{\mathcal{E}}_{h}^{l+1}}_{1,h}+C\tau\sum_{n=1}^{l}\norm{\hat{\mathcal{E}}_{h}^{n}}_{1,h}^{2}.
	    \end{eqnarray}
	    Employing the boundedness property of the bilinear form (see \eqref{boundedness}) and Young's inequality, we achieve
	    \begin{eqnarray}\label{e9}
	    	&&\left|-\tau^{2}\sum_{n=0}^{l} b(\hat{\mathcal{E}}_{h}^{n+\frac{1}{2}},\partial_{\tau}\hat{\mathcal{E}}_{h}^{n})\right|\nonumber\\
	    	&&=\tau\left|\sum_{n=0}^{l} b(\hat{\mathcal{E}}_{h}^{n+\frac{1}{2}},\hat{\mathcal{E}}_{h}^{n+1}-\hat{\mathcal{E}}_{h}^{n})\right|\nonumber\\
	    	&&\le\tau\sum_{n=0}^{l} \norm{\hat{\mathcal{E}}_{h}^{n+\frac{1}{2}}}_{1,h}\norm{\hat{\mathcal{E}}_{h}^{n+1}-\hat{\mathcal{E}}_{h}^{n}}_{1,h}\nonumber\\
	    	&&\le C\tau\sum_{n=0}^{l} \norm{\hat{\mathcal{E}}_{h}^{n+\frac{1}{2}}}_{1,h}^{2}+C\tau\sum_{n=0}^{l}\norm{\hat{\mathcal{E}}_{h}^{n+1}-\hat{\mathcal{E}}_{h}^{n}}_{1,h}^{2}\nonumber\\
	    	&&\le C\tau\sum_{n=0}^{l} \norm{\hat{\mathcal{E}}_{h}^{n+\frac{1}{2}}}_{1,h}^{2}+C\tau\sum_{n=0}^{l}\norm{\hat{\mathcal{E}}_{h}^{n+1}}_{1,h}^{2}+C\tau\sum_{n=0}^{l}\norm{\hat{\mathcal{E}}_{h}^{n}}_{1,h}^{2}\nonumber\\
	    		&&\le C\tau\sum_{n=0}^{l} \norm{\hat{\mathcal{E}}_{h}^{n+\frac{1}{2}}}_{1,h}^{2}+C\tau\sum_{n=0}^{l+1}\norm{\hat{\mathcal{E}}_{h}^{n}}_{1,h}^{2}.
	    \end{eqnarray}
	    Next, set
	    \begin{eqnarray}\label{e10*}
	    	R_{n}=\hat{I}_{h}^{k}\left(\sum_{j=0}^{n-1}\int_{t_{j}}^{t_{j+1}}(p(t_{j+\frac{1}{2}})-p(s))ds\right).
	    \end{eqnarray}
	    Hence,
	      \begin{eqnarray}\label{e10**}
	    	R_{n}-R_{n-1}&=&\hat{I}_{h}^{k}\left(\sum_{j=0}^{n-1}\int_{t_{j}}^{t_{j+1}}(p(t_{j+\frac{1}{2}})-p(s))ds\right)\nonumber\\
	    	&&-\hat{I}_{h}^{k}\left(\sum_{j=0}^{n-2}\int_{t_{j}}^{t_{j+1}}(p(t_{j+\frac{1}{2}})-p(s))ds\right)\nonumber\\
	    	&=&\hat{I}_{h}^{k}\left(\int_{t_{n-1}}^{t_{n}}(p(t_{n-\frac{1}{2}})-p(s))ds\right).
	    \end{eqnarray}
	    We use the above two identities \eqref{e10*}-\eqref{e10**} while deriving the subsequent result.
	    It follows from linearity of $b(\cdot,\cdot)$, $\hat{I}_{h}^{k}(\cdot)$ that
	    \begin{eqnarray}\label{e10***}
	    	&&\left|2\tau\sum_{n=0}^{l}\sum_{j=0}^{n-1}\int_{t_{j}}^{t_{j+1}}b(\hat{I}_{h}^{k}(p(t_{j+\frac{1}{2}}))-\hat{I}_{h}^{k}(p(s)),\partial_{\tau}\hat{\mathcal{E}}_{h}^{n})ds\right|\nonumber\\
	    	&&=\left|2\tau\sum_{n=0}^{l}b(R_{n},\partial_{\tau}\hat{\mathcal{E}}_{h}^{n})\right|\nonumber\\
	    	&&=2\left|\sum_{n=1}^{l+1}b(R_{n-1},\hat{\mathcal{E}}_{h}^{n})-\sum_{n=0}^{l}b(R_{n},\hat{\mathcal{E}}_{h}^{n})\right|\nonumber\\
	    	&&=2\left|b(R_{l},\hat{\mathcal{E}}_{h}^{l+1})+\sum_{n=1}^{l}b(R_{n-1}-R_{n},\hat{\mathcal{E}}_{h}^{n})\right|.
	    \end{eqnarray}
	    Hence from \eqref{e10***}, using boundedness property of bilinear form (see \eqref{boundedness}), identities \eqref{e10*}-\eqref{e10**}, boundedness of the global projector (see Lemma \eqref{boundedprojector}), triangle inequality, Lemma \eqref{Taylor1}, Cauchy-Schwarz inequality, Young's inequality
	    \begin{eqnarray}\label{e10}
	     	&&\left|2\tau\sum_{n=0}^{l}\sum_{j=0}^{n-1}\int_{t_{j}}^{t_{j+1}}b(\hat{I}_{h}^{k}(p(t_{j+\frac{1}{2}}))-\hat{I}_{h}^{k}(p(s)),\partial_{\tau}\hat{\mathcal{E}}_{h}^{n})ds\right|\nonumber\\
	    	&&\le2\norm{R_{l}}_{1,h}\norm{\hat{\mathcal{E}}_{h}^{l+1}}_{1,h}+2\sum_{n=1}^{l}\norm{R_{n-1}-R_{n}}_{1,h}\norm{\hat{\mathcal{E}}_{h}^{n}}_{1,h}
	    	\nonumber\\
	    	&&\le2\left\|\sum_{j=0}^{l-1}\int_{t_{j}}^{t_{j+1}}(p(t_{j+\frac{1}{2}})-p(s))ds\right\|_{H^{1}(\Omega)}\norm{\hat{\mathcal{E}}_{h}^{l+1}}_{1,h}\nonumber\\
	    	&&~~+2\sum_{n=1}^{l}\left(\left\|\int_{t_{n-1}}^{t_{n}}(p(t_{n-\frac{1}{2}})-p(s))ds\right\|_{H^{1}(\Omega)}\norm{\hat{\mathcal{E}}_{h}^{n}}_{1,h}\right)\nonumber\\
	    	&&\le2\sum_{j=0}^{l-1}\left\|\int_{t_{j}}^{t_{j+1}}(p(t_{j+\frac{1}{2}})-p(s))ds\right\|_{H^{1}(\Omega)}\norm{\hat{\mathcal{E}}_{h}^{l+1}}_{1,h}\nonumber\\
	    	&&~~+2\sum_{n=1}^{l}\left(\left\|\int_{t_{n-1}}^{t_{n}}(p(t_{n-\frac{1}{2}})-p(s))ds\right\|_{H^{1}(\Omega)}\norm{\hat{\mathcal{E}}_{h}^{n}}_{1,h}\right)\nonumber\\
	    	&&\le\frac{\tau^{2}}{2}\norm{\hat{\mathcal{E}}_{h}^{l+1}}_{1,h}\sum_{j=0}^{n-1}\int_{t_{j}}^{t_{j+1}}\norm{p_{\xi\xi}(\xi)}_{H^{1}(\Omega)}d\xi\nonumber\\
	    	&&~~+\frac{\tau^{2}}{2}\sum_{n=1}^{l}\left(\norm{\hat{\mathcal{E}}_{h}^{n}}_{1,h}\int_{t_{n-1}}^{t_{n}}\norm{p_{\xi\xi}(\xi)}_{H^{1}(\Omega)}d\xi\right)\nonumber\\
	    	&&\le\frac{\tau^{2}}{2}\norm{\hat{\mathcal{E}}_{h}^{l+1}}_{1,h}\int_{0}^{T}\norm{p_{\xi\xi}(\xi)}_{H^{1}(\Omega)}d\xi\nonumber\\
	    	&&~~+\frac{\tau^{2}}{2}\left(\sum_{n=1}^{l}\norm{\hat{\mathcal{E}}_{h}^{n}}_{1,h}^{2}\right)^{\frac{1}{2}}\left(\sum_{n=1}^{l}\left(\int_{t_{n-1}}^{t_{n}}\norm{p_{\xi\xi}(\xi)}_{H^{1}(\Omega)}d\xi\right)^{2}\right)^{\frac{1}{2}}\nonumber\\
	    	&&\le C\tau^{4}\int_{0}^{T}\norm{p_{\xi\xi}(\xi)}_{H^{1}(\Omega)}^{2}d\xi+\frac{1}{10}\norm{\hat{\mathcal{E}}_{h}^{l+1}}_{1,h}^{2}\nonumber\\
	    	&&~~+\frac{\tau^{2}}{2}\left(\sum_{n=1}^{l}\norm{\hat{\mathcal{E}}_{h}^{n}}_{1,h}^{2}\right)^{\frac{1}{2}}\left(\tau\sum_{n=1}^{l}\int_{t_{n-1}}^{t_{n}}\norm{p_{\xi\xi}(\xi)}_{H^{1}(\Omega)}^{2}d\xi\right)^{\frac{1}{2}}\nonumber\\
	    	&&\le C\tau^{4}\norm{p_{tt}}_{L^{2}(0,T;H^{1}(\Omega))}^{2}+\frac{1}{10}\norm{\hat{\mathcal{E}}_{h}^{l+1}}_{1,h}^{2}+C\tau\sum_{n=1}^{l}\norm{\hat{\mathcal{E}}_{h}^{n}}_{1,h}^{2}.
	    \end{eqnarray}
	    Next, set $
			M_{n}(p)=\hat{I}_{h}^{k}\left(\int_{t_{n}}^{t_{n+\frac{\tau}{2}}}p(t_{n+\frac{1}{2}})-p(s)ds\right)$. Then it is clear that $M_{n}(p)\in H^{1}(\Omega)$. Using definition of $M_{n}(p)$, boundedness of bilinear form (see Lemma \eqref{boundedness}) and the global projector (see Lemma \ref{boundedprojector}), Lemma \ref{Taylor2}, Cauchy-Schwarz and Young's inequality, we obtain
		\begin{eqnarray}\label{e_11}
			&&\left|2\tau\sum_{n=0}^{l}\int_{t_{n}}^{t_{n+\frac{\tau}{2}}}b(\hat{I}_{h}^{k}(p(t_{n+\frac{1}{2}}))-\hat{I}_{h}^{k}(p(s)),\partial_{\tau}\hat{\mathcal{E}}_{h}^{n})ds\right|\nonumber\\
	     	&&=2\left|\sum_{n=0}^{l}b(M_{n}(p),\hat{\mathcal{E}}_{h}^{n+1}-\hat{\mathcal{E}}_{h}^{n})\right|\nonumber\\	     		
	     	&&=2\left|\sum_{n=1}^{l+1}b(M_{n-1}(p),\hat{\mathcal{E}}_{h}^{n})-\sum_{n=0}^{l}b(M_{n}(p),\hat{\mathcal{E}}_{h}^{n})\right|\nonumber\\
	     	&&=2\left|\sum_{n=1}^{l}b(M_{n-1}(p)-M_{n}(p),\hat{\mathcal{E}}_{h}^{n})+b(M_{l}(p),\hat{\mathcal{E}}_{h}^{l+1})\right|\nonumber\\	     		&&\le2\sum_{n=1}^{l}\norm{J_{n-1}(p)-J_{n}(p)}_{H^{1}(\Omega)}\norm{\hat{\mathcal{E}}_{h}^{n}}_{1,h}+\norm{J_{l}(p)}_{H^{1}(\Omega)}\norm{\hat{\mathcal{E}}_{h}^{l+1}}_{1,h}\nonumber\\
	     	&&\le C\tau^{3}\norm{p_{tt}}_{L^{\infty}(0,T;H^1(\Omega))}\sum_{n=1}^{l}\norm{\hat{\mathcal{E}}_{h}^{n}}_{1,h}+\norm{J_{l}(p)}_{H^{1}(\Omega)}\norm{\hat{\mathcal{E}}_{h}^{l+1}}_{1,h}\nonumber\\
	     	&&\le C\tau^{3}\norm{p_{tt}}_{L^{\infty}(0,T;H^1(\Omega))}\sum_{n=1}^{l}\norm{\hat{\mathcal{E}}_{h}^{n}}_{1,h}  \nonumber\\
	     	&&~~+C\tau^{\frac{3}{2}}\left(\int_{t_{l}}^{t_{l}+\frac{\tau}{2}}\norm{p_{\xi}(\xi)}_{H^{1}(\Omega)}^{2} d\xi\right)^{\frac{1}{2}}\norm{\hat{\mathcal{E}}_{h}^{l+1}}_{1,h}	\nonumber\\
	     	&&\le C\tau^{\frac{5}{2}}\norm{p_{tt}}_{L^{\infty}(0,T;H^1(\Omega))}\left(\sum_{n=1}^{l}\norm{\hat{\mathcal{E}}_{h}^{n}}_{1,h}^{2}\right)^{\frac{1}{2}}  +C\tau^{2}\norm{p_{t}}_{L^{\infty}(0,T;H^{1}(\Omega))}^{2} \norm{\hat{\mathcal{E}}_{h}^{l+1}}\nonumber\\
	     	&&\le C\tau^{4}\norm{p_{tt}}_{L^{\infty}(0,T;H^1(\Omega))}^{2}+C\tau\sum_{n=1}^{l}\norm{\hat{\mathcal{E}}_{h}^{n}}_{1,h}^{2}\nonumber\\ &&+C\tau^{4}\norm{p_{t}}_{L^{\infty}(0,T;H^{1}(\Omega))}^{2}+ \frac{1}{10}\norm{\hat{\mathcal{E}}_{h}^{l+1}}_{1,h}^{2}.	     	
		\end{eqnarray}
		Setting
			$S_{n}=\hat{I}^{k}_{h}\left(\int_{0}^{t_{n+\frac{1}{2}}}p(s)ds\right)$, we have
		\begin{eqnarray}\label{e_12*}
			S_{n}-S_{n-1}&=&\hat{I}^{k}_{h}\left(\int_{0}^{t_{n+\frac{1}{2}}}p(s)ds\right)-\hat{I}^{k}_{h}\left(\int_{0}^{t_{n-\frac{1}{2}}}p(s)ds\right)\nonumber\\
			&=&\hat{I}^{k}_{h}\left(\int_{t_{n-\frac{1}{2}}}^{t_{n+\frac{1}{2}}}p(s)ds\right).
		\end{eqnarray}
		Using definition of $S_{n}$, triangle inequality, Lemma \ref{bounds},  equation \eqref{e_12*}, Cauchy-Schwarz inequality and Young's inequality, we achieve
	\begin{eqnarray}\label{e_12}
		&&\left|2\tau\sum_{n=0}^{l}\int_{0}^{t_{n+\frac{1}{2}}}s_{h}(\hat{I}^{k}_{h}(p(s)),\partial_{\tau}\hat{\mathcal{E}}_{h}^{n})ds\right|\nonumber\\
		&&=\left|2\tau\sum_{n=0}^{l}s_{h}(S_{n},\partial_{\tau}\hat{\mathcal{E}}_{h}^{n})\right|\nonumber\\
		&&=2\left|\sum_{n=1}^{l+1}s_{h}(S_{n-1},\hat{\mathcal{E}}_{h}^{n})-\sum_{n=0}^{l}s_{h}(S_{n},\hat{\mathcal{E}}_{h}^{n})\right|\nonumber\\
		&&=2\left|s_{h}(S_{l},\hat{\mathcal{E}}_{h}^{l+1})+\sum_{n=1}^{l}s_{h}(S_{n-1}-S_{n},\hat{\mathcal{E}}_{h}^{n})\right|\nonumber\\
		&&\le Ch^{k+1}\norm{\int_{0}^{t_{l+\frac{1}{2}}}p(s)ds}_{H^{k+2}(\Omega)}\norm{\hat{\mathcal{E}}_{h}^{l+1}}_{1,h}\nonumber\\
		&&~~+2\sum_{n=1}^{l}\left|s_{h}\left(\hat{I}^{k}_{h}\left(\int_{t_{n-\frac{1}{2}}}^{t_{n+\frac{1}{2}}}p(s)ds\right),\hat{\mathcal{E}}_{h}^{n}\right)\right|\nonumber\\
		&&\le Ch^{k+1}\norm{\int_{0}^{t_{l+\frac{1}{2}}}p(s)ds}_{H^{k+2}(\Omega)}\norm{\hat{\mathcal{E}}_{h}^{l+1}}_{1,h}\nonumber\\
		&&~~+Ch^{k+1}\sum_{n=1}^{l}\left(\left\|\int_{t_{n-\frac{1}{2}}}^{t_{n+\frac{1}{2}}}p(s)ds\right\|_{H^{k+2}(\Omega)}\norm{\hat{\mathcal{E}}_{h}^{n}}_{1,h}\right)\nonumber\\
		&&\le Ch^{k+1}\norm{\int_{0}^{t_{l+\frac{1}{2}}}p(s)ds}_{H^{k+2}(\Omega)}\norm{\hat{\mathcal{E}}_{h}^{l+1}}_{1,h}\nonumber\\
		&&~~+Ch^{k+1}\left(\sum_{n=1}^{l}\left\|\int_{t_{n-\frac{1}{2}}}^{t_{n+\frac{1}{2}}}p(s)ds\right\|_{H^{k+2}(\Omega)}^{2}\right)^{\frac{1}{2}}\left(\sum_{n=1}^{l}\norm{\hat{\mathcal{E}}_{h}^{n}}_{1,h}^{2}\right)^{\frac{1}{2}}\nonumber\\
		&&\le Ch^{k+1}\norm{\int_{0}^{t_{l+\frac{1}{2}}}p(s)ds}_{H^{k+2}(\Omega)}\norm{\hat{\mathcal{E}}_{h}^{l+1}}_{1,h}\nonumber\\
			&&~~+Ch^{k+1}\left(\tau\sum_{n=1}^{l}\int_{t_{n-\frac{1}{2}}}^{t_{n+\frac{1}{2}}}\norm{p(s)}_{H^{k+2}(\Omega)}^{2}ds\right)^{\frac{1}{2}}\left(\sum_{n=1}^{l}\norm{\hat{\mathcal{E}}_{h}^{n}}_{1,h}^{2}\right)^{\frac{1}{2}}\nonumber\\
			&&\le Ch^{2k+2}\norm{\int_{0}^{t_{l+\frac{1}{2}}}p(s)ds}_{H^{k+2}(\Omega)}^{2}+\frac{1}{10}\norm{\hat{\mathcal{E}}_{h}^{l+1}}_{1,h}\nonumber\\
			&&~~+Ch^{2k+2}\int_{t_{\frac{1}{2}}}^{t_{l+\frac{1}{2}}}\norm{p(s)}_{H^{k+2}(\Omega)}^{2}ds+C\tau\sum_{n=1}^{l}\norm{\hat{\mathcal{E}}_{h}^{n}}_{1,h}^{2}\nonumber\\
				&&\le Ch^{2k+2}\norm{p}_{L^{\infty}(0,T;H^{k+2}(\Omega))}^{2}+\frac{1}{10}\norm{\hat{\mathcal{E}}_{h}^{l+1}}_{1,h}\nonumber\\
			&&~~+Ch^{2k+2}\norm{p}_{L^{2}(0,T;H^{k+2}(\Omega))}^{2}+C\tau\sum_{n=1}^{l}\norm{\hat{\mathcal{E}}_{h}^{n}}_{1,h}^{2}.
	\end{eqnarray}
	Clubbing together the estimates \eqref{e5}-\eqref{e7}, \eqref{e8}-\eqref{e9}, \eqref{e10}-\eqref{e_11} and \eqref{e_12} then using them in \eqref{e8}, results in
			\begin{eqnarray}\label{e_13}
			&&\tau\sum_{n=0}^{l}\norm{\partial_{\tau} \mathcal{E}_{\mathcal{K}}^{n}}^{2}+\frac{1}{2}\norm{\hat{\mathcal{E}}_{h}^{l+1}}_{1,h}^{2}\nonumber\\
			&&\le C h^{2k+2}\left(\norm{p}_{L^{\infty}(0,T;H^{k+2}(\Omega))}^{2}+\norm{p}_{L^{2}(0,T;H^{k+2}(\Omega))}^{2}+\norm{p_{t}}_{L^{2}(0,T;H^{k+2}(\Omega))}^{2}\right)\nonumber\\
			&&+ C\tau^{4}\left(\norm{p_{tt}}_{L^{2}(0,T;H^{1}(\Omega))}^{2}+\norm{p_{tt}}_{L^{\infty}(0,T;H^{1}(\Omega))}^{2}+\norm{p_{t}}_{L^{\infty}(0,T;H^{1}(\Omega))}^{2}+\norm{p_{ttt}}_{L^{2}(0,T;L^{2}(\Omega))}^{2}\right)\nonumber\\
			&&+C\tau\sum_{n=0}^{l}\norm{\hat{\mathcal{E}}_{h}^{n+\frac{1}{2}}}_{1,h}^{2}+C\tau\sum_{n=1}^{l}\norm{\hat{\mathcal{E}}_{h}^{n}}_{1,h}^{2}\nonumber\\
			&&\le C h^{2k+2}\left(\norm{p}_{L^{\infty}(0,T;H^{k+2}(\Omega))}^{2}+\norm{p}_{L^{2}(0,T;H^{k+2}(\Omega))}^{2}+\norm{p_{t}}_{L^{2}(0,T;H^{k+2}(\Omega))}^{2}\right)\nonumber\\
			&&+ C\tau^{4}\left(\norm{p_{tt}}_{L^{2}(0,T;H^{1}(\Omega))}^{2}+\norm{p_{tt}}_{L^{\infty}(0,T;H^{1}(\Omega))}^{2}+\norm{p_{t}}_{L^{\infty}(0,T;H^{1}(\Omega))}^{2}+\norm{p_{ttt}}_{L^{2}(0,T;L^{2}(\Omega))}^{2}\right)\nonumber\\
			&&+C\tau\sum_{n=0}^{l+1}\norm{\hat{\mathcal{E}}_{h}^{n}}_{1,h}^{2}.
			\end{eqnarray}
			The desired estimate \eqref{maxH11} follows from \eqref{e_13} by applying discrete Gronwall's lemma.
	\end{proof}
	
	\section{\normalsize Numerical illustrations}\label{sec6}
	This section demonstrates an effective way of implementing the complete discrete HHO scheme \eqref{FDHHOscheme} via static condensation, and later we present a numerical experiment verifying the robustness and unconditional stability over very general meshes.
	\subsection{\normalsize Matrix formulations}
	We provide some numerical experiments in this section to validate the theoretical estimates. These tests include partitioning of the spatial domain using various meshes. For $p^{M}$ to be exact solution of model \eqref{model} at time $t=t_{M}=T$ and $P_{h}^{n}$ be its complete discrete counterpart evaluated from \eqref{FDHHOscheme},  we compute the exact errors $p^{M}-P_{h}^{M}$ under the energy and $L^{2}$-norms and present the order of convergence using the formula 
\[
\mathrm{order} = \frac{\log\left(\frac{E_{i+1}^{M}}{E_{i}^{M}}\right)}{\log\left(\frac{h(i+1))}{h(i)}\right)},
\]
where $E_{i}^{M}$ denotes $\norm{p^{M}-P_{h(i)}^{M}}$ or $\norm{p^{M}-P_{h(i)}^{M}}_{1,h}$, at the $i^{th}$ iteration with $h(i)$ as the corresponding mesh size.
	
	Before proceeding to verify the numerical tests, we chart out the static condensation for the complete discrete HHO scheme \eqref{FDHHOscheme}.
	
	Let $\{\phi_{1},\ldots,\phi_{N_{\mathcal{K}}}\}$ and $\{\chi_{1},\ldots,\chi_{N_{\mathcal{F}}}\}$ represent bases of broken polynomial spaces $\prod_{K\in\mathcal{K}_{h}}\mathcal{P}_{k}(K)$ and $\mathbb{F}_{h,0}^{k}$, respectively, with $N_{\mathcal{K}}$ and $N_{\mathcal{F}}$ are dimensions of the respective spaces. Denote by $(P^{n+1}_{\mathcal{K}},P^{n+1}_{\mathcal{F}})\in \mathbb{R}^{N_{\mathcal{K}}}\times\mathbb{R}^{N_{\mathcal{F}}}$ are the discrete unknowns corresponding to the numerical solution $P_{h}^{n+1}$. $F_{\mathcal{K}}^{n+1}\in\mathbb{R}^{N_{\mathcal{K}}}$ be the load term at time $t=t_{n+1}$ with entries $(f(t_{n+1}),\phi_{i})_{1\le i\le N_{\mathcal{K}}}$.
	 
	The complete discrete scheme \eqref{FDHHOscheme} can be rewritten in the following matrix formulation
	\begin{eqnarray}\label{MS1}
		&&\left(M+\frac{\tau}{2}A+\frac{\tau^{2}}{4}A\right)P^{n+1}\nonumber\\
		&&=\frac{\tau}{2}\left(F^{n+1}+F^{n}\right)+\left(M-\frac{\tau}{2}A-\frac{\tau^{2}}{4}A\right)P^{n}-\frac{\tau^{2}}{2}\sum_{j=0}^{n-1}A(U^{j+1}+U^{j}),~~~~~~~~
	\end{eqnarray}
	with
		\begin{eqnarray*}
		&&M=\begin{bmatrix}
			M_{\mathcal{K}\mathcal{K}} & \mathcal{O}_{KF}\\
			\mathcal{O}_{FK} & \mathcal{O}_{FF}
		\end{bmatrix},
		A=\begin{bmatrix}
			A_{\mathcal{K}\mathcal{K}} & A_{\mathcal{K}\mathcal{F}}\\
			A_{\mathcal{F}\mathcal{K}} & A_{\mathcal{F}\mathcal{F}}.
		\end{bmatrix},
		F^{n}=\begin{bmatrix}
			F_{\mathcal{K}}^{n}\\
			\mathcal{O}_{\mathcal{N}_{F}}
		\end{bmatrix}.
       \end{eqnarray*}
	 Here $\mathcal{O}_{KF}$, $\mathcal{O}_{FK}$, $\mathcal{O}_{FF}$, $\mathcal{O}_{\mathcal{N}_{\mathcal{F}}}$ represents the null matrix of orders $N_{\mathcal{K}}\times N_{\mathcal{F}}$, $N_{\mathcal{F}}\times N_{\mathcal{K}}$, $N_{\mathcal{F}}\times N_{\mathcal{F}}$, $N_{\mathcal{F}}\times 1$, respectively. The matrices $M_{\mathcal{K}\mathcal{K}}$, $A_{\mathcal{K}\mathcal{K}}$, $A_{\mathcal{K}\mathcal{F}}$, $A_{\mathcal{F}\mathcal{K}}$, $A_{\mathcal{F}\mathcal{F}}$ have orders $N_{\mathcal{K}}\times N_{\mathcal{K}}$, $N_{\mathcal{K}}\times N_{\mathcal{K}}$, $N_{\mathcal{K}}\times N_{\mathcal{F}}$, $N_{\mathcal{F}}\times N_{\mathcal{K}}$, and $N_{\mathcal{F}}\times N_{\mathcal{F}}$, respectively.
	 
	 Again, we can reformulate the system \eqref{MS1} as   
			\begin{eqnarray}\label{MS2}
		&&\begin{bmatrix}
			S_{\mathcal{K}\mathcal{K}} &  S_{\mathcal{K}\mathcal{F}}\\
			S_{\mathcal{F}\mathcal{K}} & S_{\mathcal{F}\mathcal{F}}
		\end{bmatrix}
		\begin{bmatrix}
			P_{\mathcal{K}}^{n+1}\\
			P_{\mathcal{F}}^{n+1}
		\end{bmatrix}
	=\begin{bmatrix}
			\mathcal{T}_{\mathcal{K}}^{n}\\
			\mathcal{T}_{\mathcal{F}}^{n}
		\end{bmatrix},
	\end{eqnarray}
	where,
	\begin{eqnarray*}
		S_{\mathcal{K}\mathcal{K}}&=&M_{\mathcal{K}\mathcal{K}}+\frac{\tau}{2}A_{\mathcal{K}\mathcal{K}}+\frac{\tau^{2}}{4}A_{\mathcal{K}\mathcal{K}},\nonumber\\
		S_{\mathcal{K}\mathcal{F}}&=&\frac{\tau}{2}A_{\mathcal{K}\mathcal{F}}+\frac{\tau^{2}}{4}A_{\mathcal{K}\mathcal{F}},\nonumber\\
				S_{\mathcal{F}\mathcal{K}}&=&\frac{\tau}{2}A_{\mathcal{F}\mathcal{K}}+\frac{\tau^{2}}{4}A_{\mathcal{F}\mathcal{K}},\nonumber\\
					S_{\mathcal{F}\mathcal{F}}&=&\frac{\tau}{2}A_{\mathcal{F}\mathcal{F}}+\frac{\tau^{2}}{4}A_{\mathcal{F}\mathcal{F}},
	\end{eqnarray*}
	and,
	\begin{eqnarray*}
		\begin{bmatrix}
			\mathcal{T}_{\mathcal{K}}^{n}\\
			\mathcal{T}_{\mathcal{F}}^{n}
		\end{bmatrix}&=&\frac{\tau}{2}\begin{bmatrix}
		F_{\mathcal{K}}^{n+1}+F_{\mathcal{K}}^{n}\\
		\mathcal{O}_{\mathcal{N}_{\mathcal{F}}}
		\end{bmatrix}+\begin{bmatrix}
		M_{\mathcal{K}\mathcal{K}}-(\frac{\tau}{2}+\frac{\tau^{2}}{4}) A_{\mathcal{K}\mathcal{K}} &  -\frac{\tau}{2}A_{\mathcal{K}\mathcal{F}}-\frac{\tau^{2}}{2}A_{\mathcal{K}\mathcal{F}}\\
		-\frac{\tau}{2}A_{\mathcal{F}\mathcal{K}}-\frac{\tau^{2}}{4}A_{\mathcal{F}\mathcal{K}} & -\frac{\tau}{2}A_{\mathcal{F}\mathcal{F}}-\frac{\tau^{2}}{4}A_{\mathcal{F}\mathcal{F}}
		\end{bmatrix}\begin{bmatrix}
		P_{\mathcal{K}}^{n}\\
		P_{\mathcal{F}}^{n}
		\end{bmatrix}\nonumber\\
		&&-\frac{\tau}{2}\begin{bmatrix}
			S_{\mathcal{K}\mathcal{K}} &  S_{\mathcal{K}\mathcal{F}}\\
			S_{\mathcal{F}\mathcal{K}} & S_{\mathcal{F}\mathcal{F}}
		\end{bmatrix}\begin{bmatrix}
		\sum_{j=0}^{n-1}(\mathcal{P}_{\mathcal{K}}^{j+1}+P_{\mathcal{K}}^{j})\\
		\sum_{j=0}^{n-1}(\mathcal{P}_{\mathcal{F}}^{j+1}+P_{\mathcal{F}}^{j})
		\end{bmatrix}.
	\end{eqnarray*} 
	From \eqref{MS2}, we get $P_{\mathcal{K}}^{n+1}=S_{\mathcal{K}\mathcal{K}}^{-1}(\mathcal{T}_{\mathcal{K}}^{n}-S_{\mathcal{K}\mathcal{F}}P_{\mathcal{F}}^{n+1})$ and substituting this in $P_{\mathcal{F}}^{n+1}=S_{\mathcal{F}\mathcal{F}}^{-1}(\mathcal{T}_{\mathcal{F}}^{n}-S_{\mathcal{F}\mathcal{K}}P_{\mathcal{K}}^{n+1})$, we achieve
	\begin{eqnarray}
		P_{\mathcal{F}}^{n+1}=(S_{\mathcal{F}\mathcal{F}}-S_{\mathcal{F}\mathcal{K}}S_{\mathcal{K}\mathcal{K}}^{-1}S_{\mathcal{K}\mathcal{F}})^{-1}(\mathcal{T}_{\mathcal{F}}^{n}-S_{\mathcal{F}\mathcal{K}}S_{\mathcal{K}\mathcal{K}}^{-1}\mathcal{T}_{\mathcal{K}}^{n}), 1\le n\le M-1.
	\end{eqnarray}
	
	\subsection{\normalsize Verification of theoretical estimates}
	
To verify the space-time order, we employ the method of manufactured solutions. Consider model \eqref{model} on the domain $\Omega\times(0,T)=(0,1)^{2}\times(0,1)$ having the exact solution as $p=\exp(-t)\sin(\pi x)\sin(\pi y)$. The forcing term $(f)$ and the initial function $(g)$ can be extracted from the given choice of $p$. We discretize the spatial domain using three different meshes: triangular ($\mathcal{K}_{h}^{1}$), quadrilateral ($\mathcal{K}_{h}^{2}$) and polygonal ($\mathcal{K}_{h}^{3}$) (see Fig \ref{Mesh}). The numerical convergence of the error is presented in Tables \ref{Tab1}-\ref{Tab3}. For the case $k=0$ in all the three tables, we have chosen $\tau=\sqrt{h}$ and for other $k\in\{1,2\}$, we have selected $\tau=10h^{\frac{k+1}{2}}$. It is clearly inferred from the information in Tables \ref{Tab1}-\eqref{Tab3} that we achieve $\mathcal{O}(h^{k+1})$ convergence for the error under the $L^{2}$ and $H^{1}$ like norms. These findings justify the theoretical error estimates.
	 
	 	\begin{figure}[!ht]
	 	\begin{subfigure}{.3\textwidth}
	 		\centering
	 		\includegraphics[width=4.50cm, height=4cm]{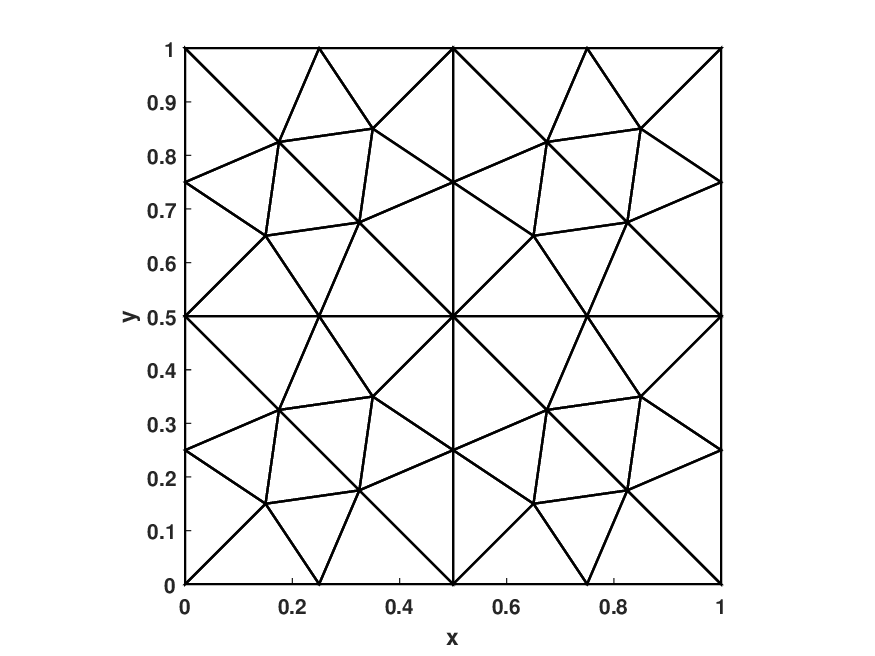}  	  		
	 	\end{subfigure}
	 	\begin{subfigure}{.3\textwidth}
	 		\centering
	 		\includegraphics[width=4.50cm, height=4cm]{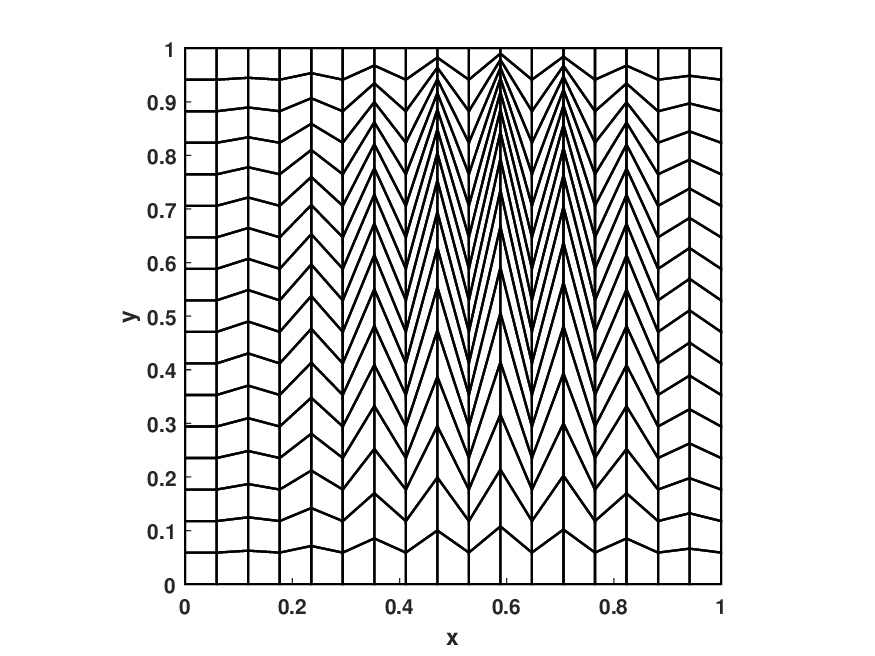}  	
	 	\end{subfigure}
	 		\begin{subfigure}{.3\textwidth}
	 		\centering
	 		\includegraphics[width=4.50cm, height=4cm]{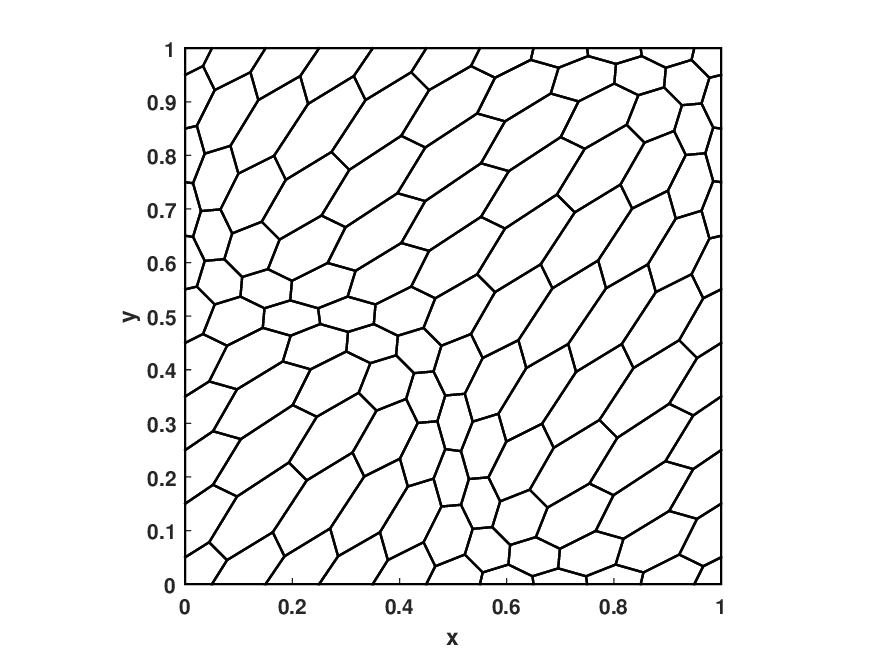}  	  		
	 	\end{subfigure}
	 	\caption{Initial triangular ($\mathcal{K}_{h}^{1}$), kershaw ($\mathcal{K}_{h}^{2}$) and polygonal meshes ($\mathcal{K}_{h}^{3}$) (from left to right).}
	 	\label{Mesh}
	 \end{figure}
	 
	 	 	\begin{table}[!ht]
	 	\setlength{\tabcolsep}{10pt}
	 	\renewcommand{\arraystretch}{1.4}
	 	\caption{Convergence of error profiles for in $L^{2}$ and discrete energy norms on triangular meshes $\mathcal{K}_{h}^{1}$. }
	 	\centering
	 	\begin{tabular}{l l l l l l}
	 		\hline
	 		$k$ & ~~~~~~~~~$h$ &  $\norm{ p^{M}-P_{h}^{M}}$ & order & ~$\norm{ p^{M}-P_{h}^{M}}_{1,h}$ & order\\
	 		\hline
	 		&~~~$3.180e-02$&~ 8.141e-03 &        ---     &~~ 1.708e-01     &    ---\\
	 		$0$&~~~$1.590e-02$&~ 1.898e-03 & 2.10 &~~ 7.126e-02 &  1.26\\
	 		&~~~$7.950e-03$&~ 4.288e-04 & 2.15 &~~ 3.340e-02 & 1.09\\
	 		&~~~$3.975e-03$&~ 7.472e-05 & 2.52 &~~ 1.636e-02 & 1.00\\
	 		\hline    
	 		&~~~$3.180e-02$&~ 1.403e-03 &        ---     &~~ 2.213e-02     &    ---\\
	 		$1$&~~~$1.590e-02$&~ 4.777e-04 & 1.55 &~~ 7.558e-03 &  1.55\\
	 		&~~~$7.950e-03$&~ 1.033e-04 & 2.21 &~~ 1.885e-03 & 2.00\\
	 		&~~~$3.975e-03$&~ 2.338e-05 & 2.14  &~~ 3.553e-04 & 2.41\\
	 		\hline     
	 		&~~~$3.180e-02$&~ 7.237e-05 &        ---     &~~ 2.898e-03     &    ---\\
	 		$2$&~~~$1.590e-02$&~ 7.912e-06 & 3.19 &~~ 3.947e-04 &  2.88\\
	 		&~~~$7.950e-03$&~ 8.448e-07 & 3.23 &~~ 5.035e-05 & 2.97\\
	 		&~~~$3.975e-03$&~ 1.005e-07 & 3.07 &~~ 6.318e-06 & 2.99\\
	 		\hline    
	 	\end{tabular}            
	 	\label{Tab1}
	 \end{table}

	 	\begin{table}[!ht]
	 	\setlength{\tabcolsep}{10pt}
	 	\renewcommand{\arraystretch}{1.4}
	 	\caption{Convergence of error profiles for in $L^{2}$ and discrete energy norms on kershaw meshes $\mathcal{K}_{h}^{2}$. }
	 	\centering
	 	\begin{tabular}{l l l l l l}
	 		\hline
	 		$k$ & ~~~~~~~~~$h$ &  $\norm{ p^{M}-P_{h}^{M}}$ & order & ~$\norm{ p^{M}-P_{h}^{M}}_{1,h}$ & order\\
	 		\hline
	 		&~~~$1.623e-02$&~ 3.956e-03 &        ---     &~~ 1.634e-01     &    ---\\
	 		$0$&~~~$8.957e-03$&~ 1.460e-03 & 1.68 &~~ 6.100e-02 &  1.66\\
	 		&~~~$6.119e-03$&~ 6.798e-04 & 2.01 &~~ 3.629e-02 & 1.36\\
	 		&~~~$4.638e-03$&~ 3.916e-04 & 1.99 &~~ 2.579e-02 & 1.23\\
	 		\hline    
	     	&~~~$1.623e-02$&~ 6.045e-04  &        ---     &~~ 1.586e-02     &    ---\\
	        $1$&~~~$8.957e-03$&~ 1.131e-04 & 2.82 &~~ 2.494e-03 &  3.11\\
	        &~~~$6.119e-03$&~ 5.728e-05 & 1.79 &~~ 1.039e-03 & 2.30\\
	        &~~~$4.638e-03$&~ 3.273e-05 & 2.02 &~~ 5.400e-04 & 2.36\\
	        \hline     
	 	   &~~~$1.623e-02$&~ 1.115e-05 &        ---     &~~ 1.310e-03     &    ---\\
	 	   $2$&~~~$8.957e-03$&~ 1.172e-06 & 3.79 &~~ 1.301e-04 &  3.89\\
	 	   &~~~$6.119e-03$&~ 3.697e-07 & 3.03 &~~ 3.862e-05 & 3.19\\
	 	   &~~~$4.638e-03$&~ 1.585e-07  & 3.06 &~~ 1.605e-05 & 3.17\\
	 	   \hline    
	 	\end{tabular}            
	 	\label{Tab2}
	 \end{table}
	 
	 		\begin{table}[!ht]
	 		\setlength{\tabcolsep}{10pt}
	 		\renewcommand{\arraystretch}{1.4}
	 		\caption{Convergence of error profiles for in $L^{2}$ and discrete energy norms on hexagonal meshes $\mathcal{K}_{h}^{3}$. }
	 		\centering
	 		\begin{tabular}{l l l l l l}
	 			\hline
	 			$k$ & ~~~~~~~~~$h$ &  $\norm{ p^{M}-P_{h}^{M}}$ & order & ~$\norm{ p^{M}-P_{h}^{M}}_{1,h}$ & order\\
	 			\hline
	 			&~~~$2.831e-02$&~ 2.935e-03 &        ---     &~~ 1.125e-01     &    ---\\
	 			$0$&~~~$1.433e-02$&~ 1.385e-03 & 1.10 &~~ 6.052e-02 &  0.91\\
	 			&~~~$7.206e-03$&~ 3.662e-04 & 1.94 &~~ 3.007e-02 & 1.02\\
	 			&~~~$3.606e-03$&~ 1.235e-04 & 1.57 &~~ 1.507e-02 & 1.00\\
	 			\hline    
	 			&~~~$2.831e-02$&~ 1.032e-03 &        ---     &~~ 1.252e-02     &    ---\\
	 		    $1$&~~~$1.433e-02$&~ 3.290e-04 & 1.68 &~~ 3.744e-03 &  1.77\\
	 	     	&~~~$7.206e-03$&~ 8.062e-05 & 2.05 &~~ 9.606e-04 & 1.98\\
	 	    	&~~~$3.606e-03$&~ 2.010e-05 & 2.01 &~~ 2.437e-04 & 1.98\\
	 	    	\hline    
	 			&~~~$2.831e-02$&~ 3.923e-05 &        ---     &~~ 1.028e-03     &    ---\\
	 			$2$&~~~$1.433e-02$&~ 4.586e-06 & 3.15  &~~ 1.614e-04 &  2.72\\
	 			&~~~$7.206e-03$&~ 5.864e-07 & 2.99 &~~ 2.080e-05 & 2.98\\
	 			&~~~$3.606e-03$&~ 7.380e-08 & 2.99 &~~ 2.663e-06 & 2.97\\
	 			\hline    
	 		\end{tabular}            
	 		\label{Tab3}
	 	\end{table}
	 
	 Selecting the time-step $\tau=10h^{\frac{3}{2}}$,  from the log-log plots in Figure \ref{L2loglog} we observe $\mathcal{O}(h^{3})$ convergence of the error under the $L^{2}$-norm while employing $k=1$ as the degree of locally approximating polynomials. This suggests that we gain an extra convergence order by selecting the time step appropriately. Justifying these phenomena will require a specifically tailored Ritz-Volterra projection for the HHO method, and we leave it for future work.    
	 	\begin{figure}[!ht]
	 	\begin{subfigure}{.5\textwidth}
	 		\centering
	 		\includegraphics[width=6.00cm, height=5.5cm]{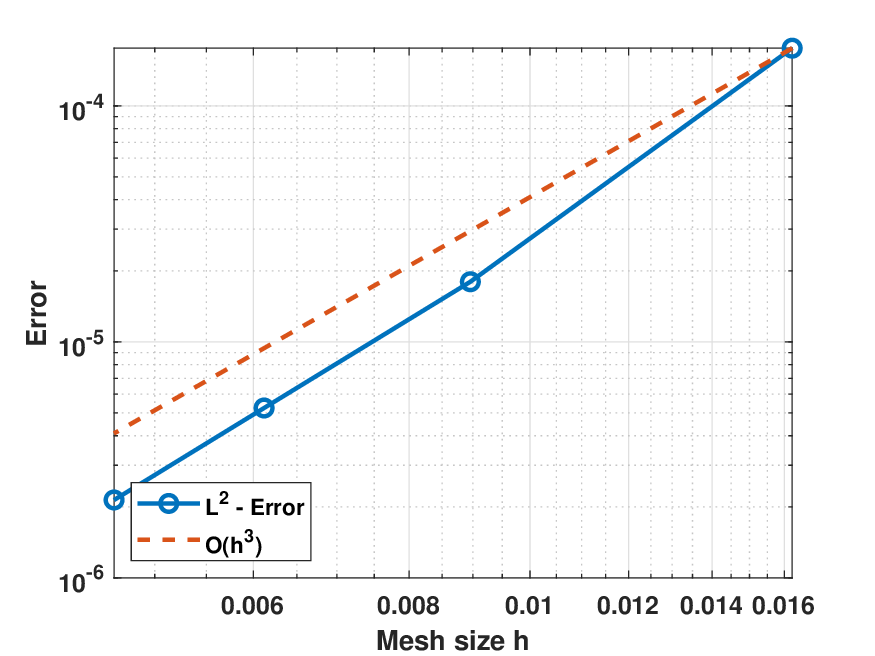}  	  		
	 	\end{subfigure}
	 	\begin{subfigure}{.5\textwidth}
	 		\centering
	 		\includegraphics[width=6.00cm, height=5.5cm]{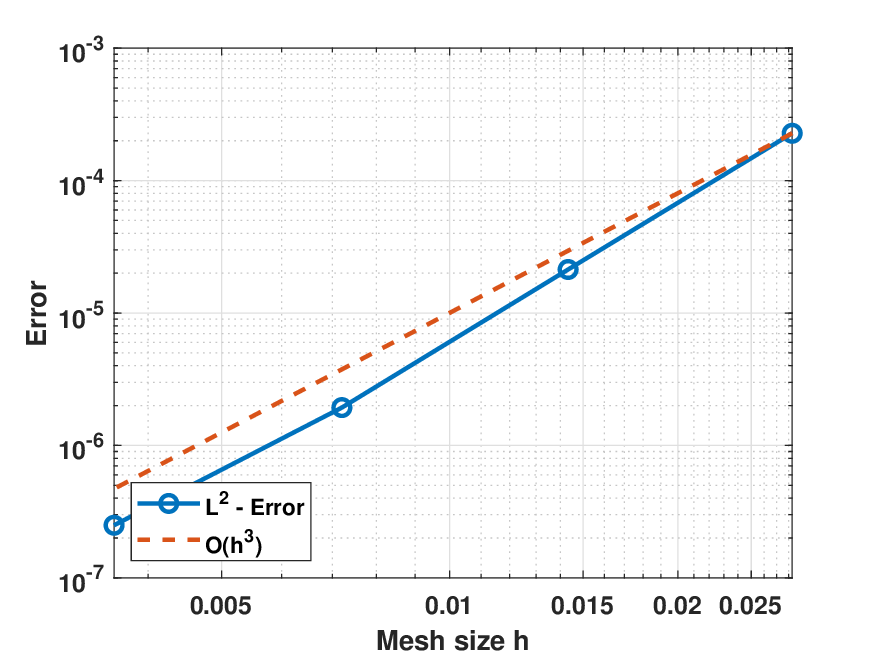}  	
	 	\end{subfigure}
	 	\caption{Convergence rates of the error under the the $L^{2}$ norm in kershaw mesh (left) and polygonal mesh (right).}
	 	\label{L2loglog}
	 \end{figure}
	 
	 Next, we provide the convergence results in the time direction. By fixing $k=1$, $T=1$, $h=3.975e-03$ and varying $\tau$, we achieve second-order convergence for the error under the $L^{2}$ norm. The corresponding results are illustrated in the log-log plot provided in Figure \ref{L2Time} which indicate that the numerical results are well in accordance with the theoretical estimates.
	 \begin{figure}[!ht]
	 \centering
	 \includegraphics[width=6.00cm, height=5.5cm]{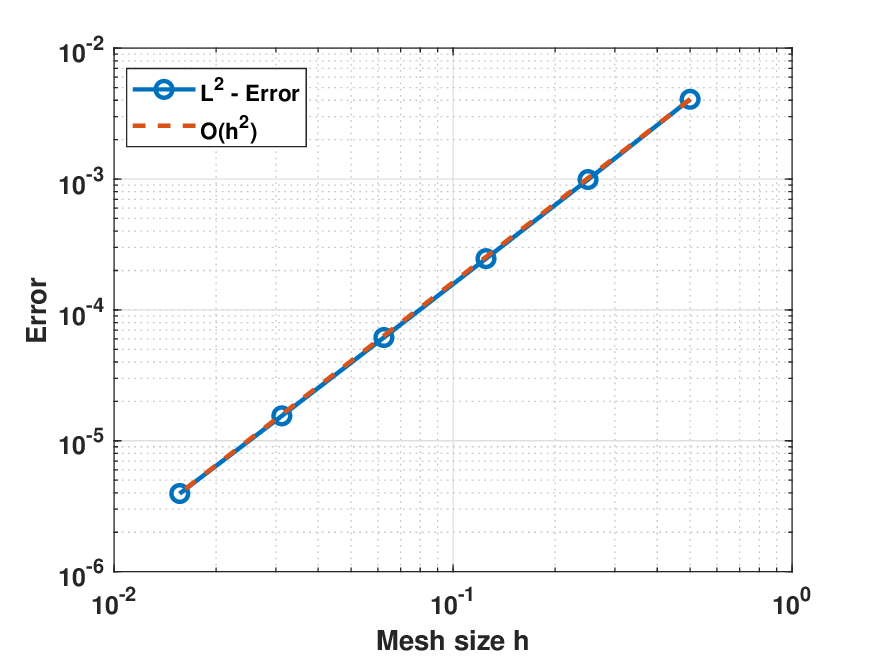}  
     \caption{Convergence rates of the error under the the $L^{2}$ norm in the triangular mesh ($\mathcal{K}_{h}^{1}$) with fixed $h=3.975e-03$ and $T=1$ .}
     \label{L2Time}
     \end{figure}
	
   \section{\normalsize Conclusion}\label{sec7}
		
		We propose a rigorous error analysis of a hybrid high-order method applied to linear PIDEs on polygonal meshes. The stability and convergence of the spatially discrete scheme are established. Next, the complete discrete Crank-Nicolson HHO method is presented, where its stability estimate under the $L^{2}$-norm is deduced. Convergence rates of $\mathcal{O}(h^{k+1}+\tau^{2})$ is derived under discretely defined $l^{2}(0,T;H^{1}(\Omega))$, $k\ge 0$ is the degree of local polynomial approximation,  and $l^{\infty}(0,T;H^{1}(\Omega))$ like norms. These outcomes exceed the convergence rates as provided by conventional FEMs. Although we observe same convergence rates for the error under the $l^{\infty}(0,T;L^{2}(\Omega))$ norm, as obtained in the energy norms, but numerically it is revealed that by setting  $\tau=\mathcal{O}(h^{\frac{k+2}{2}})$ we achieve $\mathcal{O}(h^{k+2})$ convergence in the $L^{2}$ norm. Such results for the $L^{2}$ norm are left for future work, as a specifically designed Ritz-Volterra projection for HHO method is required. 
		
	\section*{\normalsize Funding} The author is grateful to the financial support received through the Prime Minister's Research Fellowship (PMRF), Ministry of Education, Government of India (Ref: PMRF ID--1902166).

	\section*{\normalsize Data Availability}
	This manuscript has no associated data.
	\section*{\normalsize Declarations}
	
	\subsection*{\normalsize Conflict of interest} The author declares no competing interests.

	\section{\normalsize Appendix}
	\begin{lemma}\label{Taylor1}
		For any $p\in W^{2,1}(0,T;H^{1}(\Omega))$, we have the following estimate.
		\begin{equation*}
			\left\|\int_{t_{j}}^{t_{j+1}}	(p(t_{j+\frac{1}{2}})-p(s))ds\right\|_{H^{1}(\Omega)}\le\frac{\tau^{2}}{4}\int_{t_{j}}^{t_{j+1}}\norm{p_{\xi\xi}(\xi)}_{H^{1}(\Omega)}d\xi.
		\end{equation*}
	\end{lemma}
	Here $j\in\{0,1,2,\ldots M-1\}$.
	\begin{proof}
		Applying Taylor series expansion of the function $p(s)$ around the point $t=t_{j+\frac{1}{2}}$, we have
		\begin{eqnarray}\label{tal0}
			&&\int_{t_{j}}^{t_{j+1}}	(p(t_{j+\frac{1}{2}})-p(s))ds\nonumber\\
			&&=\int_{t_{j}}^{t_{j+1}}(t_{j+\frac{1}{2}}-s)p_{s}(t_{j+\frac{1}{2}})ds-\int_{t_{j}}^{t_{j+1}}\int_{t_{j+\frac{1}{2}}}^{s}(s-\xi)p_{\xi\xi}(\xi)d\xi ds\nonumber\\
			&&=\int_{t_{j}}^{t_{j+1}}\int_{s}^{t_{j+\frac{1}{2}}}(s-\xi)p_{\xi\xi}(\xi)d\xi ds\nonumber\\
			&&=\int_{t_{j}}^{t_{j+\frac{1}{2}}}\int_{s}^{t_{j+\frac{1}{2}}}(s-\xi)p_{\xi\xi}(\xi)d\xi ds-\int_{t_{j+\frac{1}{2}}}^{t_{j+1}}\int_{t_{j+\frac{1}{2}}}^{s}(s-\xi)p_{\xi\xi}(\xi)d\xi ds\nonumber\\
			&:&=M_{1}+M_{2}.            
		\end{eqnarray}
		We now find bounds for $M_{i}$, $i=1,2$.
		
		Using triangle inequality and change in order of integration, we determine
		\begin{eqnarray}\label{tal1}
			\norm{M_{1}}_{H^{1}(\Omega)}&\le&\int_{t_{j}}^{t_{j+\frac{1}{2}}}\int_{s}^{t_{j+\frac{1}{2}}}|s-\xi|\norm{p_{\xi\xi}(\xi)}_{H^{1}(\Omega)}d\xi ds\nonumber\\
			&\le&\frac{\tau}{2}\int_{t_{j}}^{t_{j+\frac{1}{2}}}\int_{s}^{t_{j+\frac{1}{2}}}\norm{p_{\xi\xi}(\xi)}_{H^{1}(\Omega)}d\xi ds\nonumber\\
			&=&\frac{\tau}{2}\int_{t_{j}}^{t_{j+\frac{1}{2}}}\norm{p_{\xi\xi}(\xi)}_{H^{1}(\Omega)}\left(\int_{t_{j}}^{\xi}ds\right) d\xi\nonumber\\
			&\le&\frac{\tau^{2}}{4}\int_{t_{j}}^{t_{j+\frac{1}{2}}}\norm{p_{\xi\xi}(\xi)}_{H^{1}(\Omega)}d\xi.
		\end{eqnarray}
		Arguing similarly as above, we get	
		\begin{eqnarray}\label{tal2}
			\norm{M_{2}}_{H^{1}(\Omega)}&\le&\int_{t_{j+\frac{1}{2}}}^{t_{j+1}}\int_{t_{j+\frac{1}{2}}}^{s}|s-\xi|\norm{p_{\xi\xi}(\xi)}_{H^{1}(\Omega)}d\xi ds\nonumber\\
			&\le&\frac{\tau}{2}\int_{t_{j+\frac{1}{2}}}^{t_{j+1}}\int_{t_{j+\frac{1}{2}}}^{s}\norm{p_{\xi\xi}(\xi)}_{H^{1}(\Omega)}d\xi ds\nonumber\\
			&=&\frac{\tau}{2}\int_{t_{j+\frac{1}{2}}}^{t_{j+1}}\norm{p_{\xi\xi}(\xi)}_{H^{1}(\Omega)}\left(\int_{\xi}^{t_{j+1}}ds\right) d\xi\nonumber\\
			&\le&\frac{\tau^{2}}{4}\int_{t_{j+\frac{1}{2}}}^{t_{j+1}}\norm{p_{\xi\xi}(\xi)}_{H^{1}(\Omega)}d\xi.
		\end{eqnarray}
		Using the bounds \eqref{tal1}-\eqref{tal2} in \eqref{tal0}, we get the desired estimate.
	\end{proof}
	\begin{lemma}\label{Taylor2}
	For any $p\in W^{2,\infty}(0,T;H^{1}(\Omega))$, denote by 
	\begin{eqnarray*}
		J_n(p):= \int_{t_n}^{t_{n}+\frac\tau2} \big(p(t_{n+\frac12}) - p(s)\big)\,ds.
	\end{eqnarray*}
   Then	the following bound holds true.
	\begin{eqnarray*}
	\norm{J_{n}(p)}_{H^{1}(\Omega)}&\le& C\tau^{\frac{3}{2}}\left(\int_{t_{n}}^{t_{n}+\frac{\tau}{2}}\norm{p_{\xi}(\xi)}_{H^{1}(\Omega)}^{2} d\xi\right)^{\frac{1}{2}},\nonumber\\
	\norm{J_{n-1}(p)-J_{n}(p)}_{H^{1}(\Omega)}	&\le&C\tau^{3}\norm{p_{tt}}_{L^{\infty}(0,T;H^1(\Omega))}. 
	\end{eqnarray*}
		Here $n\in\{0,1,2,\ldots M-1\}$.
	\end{lemma}
	\begin{proof}
		By change in order of integration, triangle inequality and Cauchy-Schwarz inequality, we derive
		\begin{eqnarray*}
			\norm{J_{n}(p)}_{H^{1}(\Omega)}=&&\left\|\int_{t_{n}}^{t_{n}+\frac{\tau}{2}}(p(t_{n+\frac{1}{2}})-p(s))ds\right\|_{H^{1}(\Omega)}\nonumber\\			&&=\left\|\int_{t_{n}}^{t_{n}+\frac{\tau}{2}}\int_{s}^{t_{n}+\frac{\tau}{2}} p_{\xi}(\xi)d\xi ds\right\|_{H^{1}(\Omega)}\nonumber\\
			&&=\left\|\int_{t_{n}}^{t_{n}+\frac{\tau}{2}}p_{\xi}(\xi)
			\left(\int_{t_{n}}^{\xi}ds\right) d\xi \right\|_{H^{1}(\Omega)}\nonumber\\
			&&=\left\|\int_{t_{n}}^{t_{n}+\frac{\tau}{2}}(\xi-t_{n})p_{\xi}(\xi) d\xi \right\|_{H^{1}(\Omega)}\nonumber\\
	    	&&\le\int_{t_{n}}^{t_{n}+\frac{\tau}{2}}|\xi-t_{n}|\norm{p_{\xi}(\xi)}_{H^{1}(\Omega)} d\xi\nonumber\\
	    	&&\le\left(\int_{t_{n}}^{t_{n}+\frac{\tau}{2}}|\xi-t_{n}|^{2}d\xi\right)^{\frac{1}{2}}\left(\int_{t_{n}}^{t_{n}+\frac{\tau}{2}}\norm{p_{\xi}(\xi)}_{H^{1}(\Omega)}^{2} d\xi\right)^{\frac{1}{2}}\nonumber\\
	    	&&\le C\tau^{\frac{3}{2}}\left(\int_{t_{n}}^{t_{n}+\frac{\tau}{2}}\norm{p_{\xi}(\xi)}_{H^{1}(\Omega)}^{2} d\xi\right)^{\frac{1}{2}}.
		\end{eqnarray*}
		This completes proof of the first result.
		
		Further note that we can write (see the above proof)
	\begin{eqnarray*}
    	J_{n}=\int_{t_n}^{t_{n+\frac12}}(\xi-t_{n}) p_{\xi}(\xi)\,d\xi.
    \end{eqnarray*}
    Hence,
    \begin{eqnarray*}
    	J_{n-1}=\int_{t_{n-1}}^{t_{n-\frac12}}(\xi-t_{n-1}) p_{\xi}(\xi)d\xi=\int_{t_{n-1}}^{t_{n-\frac12}}(\xi-t_{n}+\tau) p_{\xi}(\xi)d\xi.
    \end{eqnarray*}
		Thus, we have
		\begin{eqnarray}\label{ap1}
		   J_{n-1}-J_n
			&=&
			\int_{t_{n-1}}^{t_{n-\frac12}} (\xi - t_n)p_{\xi}(\xi)d\xi
		-	\int_{t_n}^{t_{n+\frac12}} (\xi - t_n)p_{\xi}(\xi)d\xi\nonumber\\
			&& + \tau \int_{t_{n-1}}^{t_{n-\frac12}} p_{\xi}(\xi)d\xi.
		\end{eqnarray}
		By Fundamental Theorem of calculus, we get
		\begin{eqnarray*}
				p_{\xi}(\xi)=
			p_{\xi}(t_n) + \int_{t_n}^{\xi} p_{\eta\eta}(\eta)\,d\eta.
		\end{eqnarray*}
		Substituting the above relation in \eqref{ap1}, we obtain
			\begin{eqnarray}\label{ap2}
			J_{n-1}-J_n
			&=&
			\int_{t_{n-1}}^{t_{n-\frac12}} (\xi - t_n)
			p_{\xi}(t_n)d\xi + 	\int_{t_{n-1}}^{t_{n-\frac12}}(\xi - t_n) \left(\int_{t_n}^{\xi} p_{\eta\eta}(\eta)\,d\eta\right)d\xi
			\nonumber\\
			&&-
			\int_{t_n}^{t_{n+\frac12}} (\xi - t_n)	p_{\xi}(t_n)d\xi - \int_{t_n}^{t_{n+\frac12}} (\xi - t_n)\left(\int_{t_n}^{\xi} p_{\eta\eta}(\eta)\,d\eta\right)d\xi\nonumber\\
			&& + \tau \int_{t_{n-1}}^{t_{n-\frac12}}p_{\xi}(t_n)d\xi + \tau \int_{t_{n-1}}^{t_{n-\frac12}} \int_{t_n}^{\xi} p_{\eta\eta}(\eta)d\eta d\xi.\nonumber\\
			&=& 	\int_{t_{n-1}}^{t_{n-\frac12}}(\xi - t_n) \left(\int_{t_n}^{\xi} p_{\eta\eta}(\eta)d\eta\right)d\xi+ \tau \int_{t_{n-1}}^{t_{n-\frac12}} \int_{t_n}^{\xi} p_{\eta\eta}(\eta)\,d\eta d\xi\nonumber\\
			&& - \int_{t_n}^{t_{n+\frac12}} (\xi - t_n)\left(\int_{t_n}^{\xi} p_{\eta\eta}(\eta)\,d\eta\right)d\xi:=\sum_{i=1}^{3}Q_{i}.
		\end{eqnarray}
		We now estimate each of terms $Q_{i}$, $i=1,2,3$, under the $H^{1}(\Omega)$-norm.
		
		In the term $J_{n-1}$, we have $t_{n-1}\le\xi\le t_{n-\frac{1}{2}}<t_{n}$.
		Next from triangle inequality, repeated applications of Cauchy-Schwarz inequality and change in order of integration results in
			\begin{eqnarray}\label{ap3}
			\norm{Q_{1}}_{H^{1}(\Omega)}&\le&\int_{t_{n-1}}^{t_{n-\frac12}} |\xi - t_n|
			\left\|\int_{\xi}^{t_{n}} p_{\eta\eta}(\eta)\,d\eta \right\|_{H^1(\Omega)} d\xi\nonumber\\
			&\le&\int_{t_{n-1}}^{t_{n-\frac12}} |\xi - t_n|^{\frac{3}{2}}\left(\int_{\xi}^{t_{n}} \norm{p_{\eta\eta}(\eta)}_{H^1(\Omega)}^{2} d\eta\right)^{\frac{1}{2}}d\xi\nonumber\\
			&\le&\left(\int_{t_{n-1}}^{t_{n-\frac12}} |\xi - t_n|^{3}d\xi\right)^{\frac{1}{2}}\left(\int_{t_{n-1}}^{t_{n-\frac12}}\int_{\xi}^{t_{n}} \norm{p_{\eta\eta}(\eta)}_{H^1(\Omega)}^{2}d\eta d\xi \right)^{\frac{1}{2}}\nonumber\\
				&\le&\left(\int_{t_{n-1}}^{t_{n-\frac12}} |\xi - t_n|^{3}d\xi\right)^{\frac{1}{2}}\nonumber\\
				&&\times\left(\int_{t_{n-1}}^{t_{n-\frac12}} \norm{p_{\eta\eta}(\eta)}_{H^1(\Omega)}^{2}\left(\int_{t_{n-1}}^{\eta} d\xi\right)d\eta\right.\nonumber\\
				&&~~~~~\left.+\int_{t_{n-\frac12}}^{t_{n}} \norm{p_{\eta\eta}(\eta)}_{H^1(\Omega)}^{2}\left(\int_{t_{n-1}}^{t_{n-\frac{1}{2}}} d\xi\right)d\eta \right)^{\frac{1}{2}}\nonumber\\
			&\le&C\tau^{2}\left(\int_{t_{n-1}}^{t_{n}}(\eta-t_{n-1}) \norm{p_{\eta\eta}(\eta)}_{H^1(\Omega)}^{2}d\eta \right)^{\frac{1}{2}}\nonumber\\
			&\le&C\tau^{\frac{5}{2}}\left(\int_{t_{n-1}}^{t_{n}} \norm{p_{\eta\eta}(\eta)}_{H^1(\Omega)}^{2}d\eta \right)^{\frac{1}{2}}\nonumber\\
			&\le&C\tau^{3}\norm{p_{tt}}_{L^{\infty}(0,T;H^1(\Omega))}.
		\end{eqnarray}
		 Again, by changing order of integration and Cauchy-Schwarz inequality, we deduce
		\begin{eqnarray}\label{ap4}
			\norm{Q_{2}}_{H^{1}(\Omega)}&\le&\tau \int_{t_{n-1}}^{t_{n-\frac12}} \int_{\xi}^{t_{n}} \norm{p_{\eta\eta}(\eta)}_{H^{1}(\Omega)}\,d\eta d\xi\nonumber\\
	    	&\le&	\tau \int_{t_{n-1}}^{t_{n}}(\eta-t_{n-1}) \norm{p_{\eta\eta}(\eta)}_{H^{1}(\Omega)}d\eta\nonumber\\
	    	&\le&	C\tau\left( \int_{t_{n-1}}^{t_{n}}|\eta-t_{n-1}|^{2}d\eta\right)^{\frac{1}{2}} \left( \int_{t_{n-1}}^{t_{n}}\norm{p_{\eta\eta}(\eta)}_{H^{1}(\Omega)}^{2}d\eta\right)^{\frac{1}{2}}\nonumber\\
	    	&\le&	C\tau^{\frac{5}{2}} \left( \int_{t_{n-1}}^{t_{n}}\norm{p_{\eta\eta}(\eta)}_{H^{1}(\Omega)}^{2}d\eta\right)^{\frac{1}{2}}\nonumber\\
	    	&\le&C\tau^{3}\norm{p_{tt}}_{L^{\infty}(0,T;H^1(\Omega))}.
		\end{eqnarray}
	    For the term $J_{n}$, we have $t_{n}\le\xi\le t_{n+\frac{1}{2}}$. Hence, arguing as above, we determine
		\begin{eqnarray}\label{ap5}
		\norm{Q_{3}}_{H^{1}(\Omega)}&\le&\int_{t_{n}}^{t_{n+\frac12}} |\xi - t_n|
		\left\|\int_{t_n}^{\xi} p_{\eta\eta}(\eta)\,d\eta \right\|_{H^1(\Omega)} d\xi\nonumber\\
		&\le&\int_{t_{n}}^{t_{n+\frac12}} |\xi - t_n|^{\frac{3}{2}}\left(\int_{t_n}^{\xi} \norm{p_{\eta\eta}(\eta)}_{H^1(\Omega)}^{2} d\eta\right)^{\frac{1}{2}}d\xi\nonumber\\
		&\le&\left(\int_{t_{n}}^{t_{n+\frac12}} |\xi - t_n|^{3}d\xi\right)^{\frac{1}{2}}\left(\int_{t_{n}}^{t_{n+\frac12}}\int_{t_n}^{\xi} \norm{p_{\eta\eta}(\eta)}_{H^1(\Omega)}^{2}d\eta d\xi \right)^{\frac{1}{2}}\nonumber\\
		&\le&C\tau^{2}\left(\int_{t_{n}}^{t_{n+\frac12}}(t_{n+\frac{1}{2}}-\eta) \norm{p_{\eta\eta}(\eta)}_{H^1(\Omega)}^{2}d\eta \right)^{\frac{1}{2}}\nonumber\\
		&\le&C\tau^{\frac{5}{2}}\left(\int_{t_{n}}^{t_{n+\frac12}} \norm{p_{\eta\eta}(\eta)}_{H^1(\Omega)}^{2}d\eta \right)^{\frac{1}{2}}\nonumber\\
			&\le&C\tau^{3}\norm{p_{tt}}_{L^{\infty}(0,T;H^1(\Omega))}.
		\end{eqnarray}
		Clubbing together \eqref{ap3}-\eqref{ap5} in \eqref{ap2}, we get second result.
	\end{proof}
\end{document}